\documentclass[final]{siamart1116}

\usepackage{amsfonts}
\usepackage{amsmath}
\usepackage{mathtools}
\usepackage{enumitem} 
\allowdisplaybreaks
\usepackage{graphicx}
\ifpdf
  \DeclareGraphicsExtensions{.pdf,.eps,.png,.jpg}
\else
  \DeclareGraphicsExtensions{.eps}
\fi
\graphicspath{{./pictures/}}

\usepackage{subfigure}
\usepackage{bm}

\def\L{\mathrm{L}}
\def\H{\mathrm{H}}

\renewcommand{\theequation}{\arabic{section}.\arabic{equation}}
\newsiamremark{remark}{Remark}
\numberwithin{theorem}{section}

\newcommand{\TheTitle}{``Finite element approximation for the delayed generalized Burgers-Huxley equation with weakly singular kernel: Part II Non-Conforming and DG approximation"} 
\newcommand{\TheShortTitle}{FEM for GBHE with Memory} 
\newcommand{\TheAuthors} {Sumit Mahajan and  Arbaz Khan }

\headers{\TheShortTitle}{\TheAuthors}

\ifpdf
\hypersetup{
  pdftitle={\TheTitle},
  pdfauthor={\TheAuthors}
}
\fi

\title{{\TheTitle}%
\thanks{ 
\funding{
``S. Mahajan would like to thank Ministry of Education, Government of India, for financial support (Prime Minister Research Fellowship PMRF ID : 2801816) to carry out his research work."
}}}

\author{
Sumit Mahahjan\thanks{ ``Department of Mathematics,  Indian Institute of Technology Roorkee (IITR), 
Roorkee-247667, India 
    (\email{sumit\_m@ma.iitr.ac.in})",}
   \and
   Arbaz Khan\thanks{``Corresponding author: 
 Department of Mathematics,  Indian Institute of Technology Roorkee (IITR), 
Roorkee-247667, India (\email{arbaz@ma.iitr.ac.in})".}}

\begin{document}

\maketitle

\begin{abstract}
In this paper, the numerical approximation of the generalized Burgers'-Huxley equation (GBHE) with weakly singular kernels using non-conforming methods will be presented. Specifically, we discuss two new formulations. The first formulation is based on the non-conforming finite element method (NCFEM). The other formulation is based on discontinuous Galerkin finite element methods  (DGFEM).
The wellposedness results for both formulations are proved.
  Then, a priori error estimates for both the semi-discrete and fully-discrete schemes are derived.
  Specific numerical examples, including some applications for the GBHE with weakly singular model, are discussed to validate the theoretical results.
\end{abstract}



\section{Introduction}
\label{intro}
Non-linear partial differential equations (PDEs) find numerous applications in the various fields of physics, biology, mechanics, and dynamics. As of now, solving these equations remains highly challenging, and finding solutions, whether through analytical or numerical approaches, is a complex task. The model's complexity and non-linearity pose difficulties in achieving accurate and reliable solutions. To make these complex models solvable, we frequently need to introduce different assumptions like simplifying the equations, ignoring certain factors, or estimating the solution. Although these simplifications can make the problem easier to handle, but this becomes problematic when we apply the solution to real-world problems, where accuracy and reliability are of utmost importance. One such exemplar model is the 
GBHE, which explains the interplay between convection effects, diffusion transport, and reaction mechanisms.
Our model problem is as follows: Find $u\in \Omega\times [0,T]$, such that 
\begin{align}\label{2.GBHE}
	\nonumber\mathcal{L} u(x,t) &=f(x,t),  \ (x,t)\in\Omega\times(0,T),\\ \quad u(x,t)&=0, \ (x,t) \in \partial\Omega \times (0,T),\quad 
	u(x,0)=u_0(x), \ x\in\Omega, 
\end{align}
where the domain $\Omega\subset \mathbb{R}^d (d = 2,3)$ is an open bounded simply connected convex domain and the boundary $\partial\Omega$ is  Lipschitz. $f(\cdot,\cdot)$ represents the given external forcing and the differential operator is defined as
\begin{align*}
	\nonumber\mathcal{L} u &=  \frac{\partial u}{\partial t}-\nu\Delta u+\alpha u^{\delta}\sum_{i=1}^d\frac{\partial u}{\partial x_i}- \beta u(1-u^{\delta})(u^{\delta}-\gamma)-\eta\int_{0}^{t} K(s-\tau)\Delta u(\tau) \mathrm{~d}\tau.
\end{align*}
The delayed effect of the GBHE is studied by the memory term where $\eta \geq 0$ signifies the relaxation time and $K(\cdot)$ denotes the weakly singular kernel. The parameters $\alpha>0,\delta\ge1,\beta >0$, $\gamma\in(0,1)$ and $\nu$ represent 
the advection coefficient, 
 the retardation time, 
 the reaction coefficient, the constant and the diffusion coefficient, respectively.  
 For the different choices of the parameters, the above model can be reduced to Burgers equation~\cite{JMB}, which has various applications in fluid dynamics, traffic flow, etc., or the Huxley equation~\cite {WXY}, which describes nerve pulse propagation in nerve fibres and wall motion in liquid crystals, or Fitz-Hugh-Nagumo~\cite{FHR} equation which is a reaction-diffusion equation utilized in both circuit theory and biology to describe dynamic processes~\cite{DJS}.

Numerous research studies explored analytic and numerical solutions for the 1-D version of the GBHE and similar reducible equations. Different methods are available in the literature, such as spectral methods~\cite{HAE}, hybrid spectral-collocation methods~\cite{MTD}, variational iteration methods~\cite{BMI}, Adomian decomposition method~\cite{HRRA}, homotopy analysis method~\cite{AMF}, differential transform method~\cite{JFm}, the Haar wavelength methods~\cite{CIb}, collocation methods~\cite{MRe}, and many more.However, for the higher-dimensional case (2D-3D), the performance of some NSFD methods has been studied in [32], and Ervin et al. have discussed finite element approximation by lagging the non-linearity in~\cite{EMR}. 

The global solvability of the GBHE without memory ($\eta = 0$) in 1D using conforming FEM is studied in \cite{MKH}.
 However, the fully discrete case has not been addressed there. In the following year, in \cite{KMR}, the numerical approximation using standard conforming, non-conforming, and DG approximation for the stationary counterpart in higher dimensions ($\mathbb{R}^2$ and $\mathbb{R}^3$) has been discussed under stringent conditions on parameters and given data, as stated in \cite[Theorem 3.3-3.6]{KMR}. 
 In \cite{GBHE}, the authors established the first result in the direction of the existence and uniqueness of the weak solution for the GBHE with memory.
 Moreover, the paper discusses the regularity results under different assumptions on the initial data and external forcing.
 A priori error estimates using the standard conforming finite element method (CFEM) are also given  in \cite{GBHE}. 
 
As per the author's knowledge, this work is the first contribution in the direction of the non-conforming approximation of GBHE with weakly singular kernels using CR and DG elements. Details of the significant contributions of this work are as follows:

\begin{itemize}
	\item[\textbf{--}] In this study, we propose two novel finite element discretization schemes for the GBHE equation with memory using non-conforming and DG approximation, presented in equations \eqref{2.ncweakform} and \eqref{2.dgweakform}. Specifically, we propose the new idea to handle the nonlinear convective terms. These formulations facilitate the proof of solvability, stability, and a priori error estimates without imposing any constraints on the parameters. 
	Moreover, these new schemes would also be applicable to a variety of fluid flow models for estimating the convection term.
	\item[\textbf{--}]   Due to the presence of weakly singular kernels, the analysis becomes complex due to the existence of singularities at specific points, despite the valuable insights they provide. By assuming the positive nature of the weakly singular kernel, we establish optimal convergence for the semi-discrete scheme using both CR and DG elements. 
		\item[\textbf{--}]   The significance of our work lies in providing error estimates for the fully discrete case without relying on the assumption $u_{tt}\in \L^2(0, T;\L^2(\Omega))$, which necessitates smoother boundary conditions and may not be applicable to various natural physical problems. Our analysis demonstrates the convergence of the fully discrete scheme under minimal regularity assumptions, making it suitable for convex domains or domains with $C^2$ boundaries, thereby catering to a wide range of problems.
	\item[\textbf{--}] Furthermore, we conduct numerical computations for various examples to validate the derived results. Additionally, we offer numerical evidence supporting the applicability of our proposed method to equations involving the Caputo fractional derivative and showing the spiral wave structure for the FitzHugh–Nagumo model.
\end{itemize}
Lately, the residual-based a posteriori error estimators for the 
 GBHE with memory will be discussed in \cite{GBHE3}, which is the subject of ongoing research.

The paper is organized as follows: Section \ref{2.sec1} introduces the notations used throughout the paper and outlines the regularity results from \cite{GBHE}. Section \ref{2.sec3} focuses on the numerical approximation using finite element discretization. In Section \ref{2.secCR}, we present a semi-discrete formulation that employs Crouzeix-Raviart (CR) elements in space and establishes the solvability result using Carath'eodory's existence theorem for the discrete system. Additionally, we discuss the optimal a prior error estimates achieved via finite element interpolation. The paper further delves into fully-discrete error estimates, utilizing backward Euler in time and NCFEM in space, as discussed in Section \ref{sec2.2.2}. We also present corresponding findings using DG elements, which are discussed in Section \ref{2.secDG}. Finally, Section \ref{2.secnum} examines and discusses the computational results.

\section{Finite Element Method} \setcounter{equation}{0}\label{2.sec3}
In this section, we first provide the necessary functional space and notations that are used consistently in the paper. Further, the error estimates are discussed using NCFEM and DGFEM for both semi-discrete as well as fully-discrete cases.
\subsection{Preliminaries}\label{2.sec1}
Let $\mathrm{C}_0^{\infty}(\Omega)$ be the set of infinitely differentiable functions  having compact support within the domain $\Omega.$ The spaces, $\L^p(\Omega)$ for $p\in[1,\infty],$ demonstrate the standard Lebesgue spaces and their associated norms are represented as $\|\cdot\|_{\L^p}.$ Let $\H^k(\Omega)$ be the standard Sobolev space. Specifically, the space $\H_0^1(\Omega)$ represents the closure of $\mathrm{C}_0^{\infty}(\Omega)$ with respect to $\H^1$-norm. The sum space $ X'_p = \H^{-1}(\Omega)+\L^{p'}(\Omega)$ is the dual space of the intersection space 	$ X_p = \H_0^1(\Omega)\cap\L^p(\Omega).$
We consider the kernel $K(\cdot)$ to be \emph{weakly singular positive kernel} such  that $K\in\L^1(0,T)$ and    for any $T> 0$, we have
\begin{align}\label{2.pk}
	\int_0^T \int_0^t K(t-\tau)u(\tau)u(t)\mathrm{~d}\tau \mathrm{~d}t\geq 0,\  \ \forall\ u\in \L^2(0, T).
\end{align}
The weak formulation for $u_0\in\L^{2}(\Omega)$ and $f\in\L^{2}(0,T;\H^{-1}(\Omega))$, of \eqref{2.GBHE}, for a.e. $t \in (0, T )$, is given by  
	\begin{align}\label{2.weakform}
		\nonumber(\partial_tu(t), v(t))+\nu (\nabla u(t),\nabla v)+\alpha b(u(t),u(t),v)+\eta(( K*\nabla u)(s),\nabla v)-\beta\langle c(u(t)),v\rangle&=\langle f(t),v\rangle\\
		(u(0), v(t))&=(u_0,v(t)),
	\end{align}
	for any $v\in X_{2(\delta+1)}$ where
	\begin{align*}
		b(u,v,w)= \left(u^{\delta}\sum_{i=1}^d\frac{\partial u}{ \partial x_i},w\right),
		\quad 	(c(u),v)=(u(1-u^{\delta})(u^{\delta}-\gamma),v).
\end{align*}
The existence and uniqueness of the  weak solution \eqref{2.GBHE} have been discussed in \cite{GBHE} and for the smoothness assumption on the initial data, we have the following regularity results
\begin{theorem}[Regularity]\label{2.weaksolution} Let $u$ be the solution of the weak form defined in \eqref{2.weakform}. 
	\begin{enumerate}
		\item  For $u_0\in\L^{2}(\Omega)$ and $f\in\L^{2}(0,T;\H^{-1}(\Omega))$  we have, $\partial_tu\in\L^{\frac{2(\delta+1)}{2\delta+1}}(0,T;X'_{2(\delta+1)})$ and 
		\begin{align*}
			\nonumber
			&u\in \L^{2}(0,T; \H_0^1(\Omega))\cap\mathrm{L}^{\infty}(0,T; \L^{2}(\Omega))\cap\L^{2(\delta+1)}(0,T; \L^{2(\delta+1)}(\Omega)).
		\end{align*}
		\item For $f\in\L^2(0,T;\L^2(\Omega))$ and $u_0\in X_{2(\delta+1)},$ it follows that
			\begin{align*}
				\nonumber&u\in\mathrm{L}^2(0,T;\H^2(\Omega))\cap\mathrm{L}^{\infty}(0,T;X_{2(\delta+1)})\cap\mathrm{L}^{2(\delta+1)}(0,T;\mathrm{L}^{6(\delta+1)}(\Omega)),\partial_t u\in\mathrm{L}^2(0,T;\L^2(\Omega)).
		\end{align*}	 
		\item 	If $\delta \in [1,\infty)$, for $d=2$, and $\delta\in[1,2]$ for $d=3.$ For $u_0\in\H^2(\Omega)\cap\H_0^1(\Omega)$ and $f\in\H^1(0,T;\L^2(\Omega))$,we have $$\partial_tu\in\L^{\infty}(0,T;\L^2(\Omega))\cap\L^2(0,T;\H_0^1(\Omega)).$$ Additionally, for  $u \in \L^{\infty}(0,T;\H^2(\Omega))$ we need $f\in\H^1(0,T;\L^2(\Omega))\cap \L^2(0,T;\H^1(\Omega)).$ 
	\end{enumerate}

\end{theorem}
\begin{proof}
	The above regularity result have already been established in \cite[Theorem 2.2-2.5]{GBHE}.
\end{proof}
\begin{lemma}\cite{MTM}\label{2.l1}
	There holds: 
	$$\left(\int_0^{T}\left(\int_0^sK(s-\tau)\phi(\tau)\mathrm{~d}\tau\right)^2\mathrm{~d}s\right)^{\frac{1}{2}}\leq \left(\int_0^{T}|K(s)|\mathrm{~d}s\right)\left(\int_0^{T}\phi^2(s)\mathrm{~d}s\right)^\frac{1}{2},$$
	for each $\phi\in \L^2(0,T)$ and $K \in \L^1(0,T)$ with $T>0$.
\end{lemma}
\subsection{Non-conforming Finite Element Method}
\subsubsection{Semi-discrete non-conforming FEM}\label{2.secCR}
This section is devoted to the semi-discrete Galerkin approximation of GBHE with memory using NCFEM. The domain $\Omega$ is divided into shape-regular meshes (consisting of triangular or rectangles for 2D or tetrahedron for 3D) denoted by $\mathcal{T}_h$. Let the set of edges, the interior edges, and the boundary edges of the triangulation be denoted by the symbols $\mathcal{E}_h$, $\mathcal{E}^i_h$ and $\mathcal{E}^{\partial}_h$, respectively. For a given $\mathcal{T}_h$, $C^{0}(\mathcal{T}_h)$ and $H^s(\mathcal{T}_h)$  denote the broken spaces linked with continuous and differentiable function spaces, respectively. 
Let the space of polynomials having a degree at most one be given by $\mathbb{P}_1$. The definition of the finite element space using Crouzeix-Raviart (CR) element
\begin{equation}\label{2.CRFEM11}
	V_h =\left\{v\in \L^2(\Omega) :\ \forall\  K\in \mathcal{T}_h ; v_{|_K}\in\mathbb{P}_1 \;\mbox{and}\; \int_E [|v|]=0\quad E\in\mathcal{E}\right\}.
\end{equation}
For each triangulation, we define the piecewise gradient as $\nabla_h: \H^1(\mathcal{T}_h)\rightarrow \L^2(\Omega;\mathbb{R}^d)$ with $(\nabla_h v)|_K = \nabla v|_K, \forall ~K\in \mathcal{T}_h$. In this context,
the semi-discrete weak formulation of (\ref{2.GBHE}) is given as: For each $t\in(0,T)$, find $ u_h\in V_h$ such that
\begin{align}\label{2.ncweakform}
\nonumber(\partial_tu_h, \chi)+A_{CR}(u_h(t),\chi)+\eta(( K*\nabla_h u_h)(t),\nabla_h \chi)&=( f^k,\chi),\\
(u_h(0),\chi)&=(u_h^0,\chi), \quad \quad \forall \chi \in V_h,
\end{align}
where, 
\begin{align*}
A_{CR}(u_h(t),\chi):=&  \nu a_{CR}(u_h(t),\chi) + \alpha b_{CR}(u_h(t),u_h(t),\chi)-\beta(c(u_h(t)),\chi).\
\end{align*}
with $a_{CR}(u,v) = (\nabla_h u,\nabla_h v)$ and $c(u)=u(1-u^{\delta})(u^{\delta}-\gamma).$
For the non-linear operator, if we define the operator $b_{CR}(\cdot,\cdot,\cdot)$ as in the case of conforming FEM{\cite{GBHE}}, given by
\[b_{CR}(u,v,w) =  \sum_{K\in \mathcal{T}_h}\int_{\mathcal{T}_h} u^{\delta}(x)\sum_{i=1}^{d}\frac{\partial v(x)}{\partial x_i}w(x)\mathrm{~d}x,\]
then $b_{CR}(u,u,u)\neq 0$ and using H\"older's and Young's inequality as
\begin{align*}
\alpha b_{CR}(u,u,u)
 = \frac{\alpha}{\delta+1} \Bigg(\frac{\partial u(x)}{\partial {x_i}} ,u^{\delta+1}(x)\Bigg)_{\mathcal{T}_h} \leq \epsilon\|\nabla_h u\|_{\L^2(\mathcal{T}_h)}^2 + C(\epsilon)\|u\|_{\L^{2(\delta+1)}}^{2(\delta+1)},
\end{align*}
where $\epsilon$ depends on the parameters $\alpha, \beta, \delta$.
The stability estimate as in \cite[Lemma 3.2]{GBHE} does not hold true for any choice of parameters (depends on the choice of $\epsilon$). So, to avoid the restriction on parameters, we redefine the operator as: 
For $u,w\in \H^1_0(\Omega),$
 using integration by parts in $b(\cdot, \cdot,\cdot)$, we have  
\begin{align*}
\nonumber	b(u;u,w) = \left(u^{\delta}\sum_{i=1}^d\frac{\partial u}{ \partial x_i},w\right)&= a_1 \left(u^{\delta}\sum_{i=1}^d\frac{\partial u}{ \partial x_i},w\right) + a_2 \left(u^{\delta}\sum_{i=1}^d\frac{\partial u}{ \partial x_i},w\right)\\&= a_1 \left(u^{\delta}\sum_{i=1}^d\frac{\partial u}{ \partial x_i},w\right) - \frac{a_2}{\delta+1} \left(u^{\delta}\sum_{i=1}^d\frac{\partial w}{ \partial x_i},u\right),
\end{align*}
where $a_1, a_2$ are constants chosen such that $a_1 + a_2 = 1$. In particular, take $a_1 = \frac{a_2}{\delta+1}$, so we introduce
\begin{align}\label{2.fcr}
 \nonumber b_{CR}(u;u,w) &:= \frac{1}{\delta+2}\left(u^{\delta}\sum_{i=1}^d\frac{\partial u}{ \partial x_i},w\right)_{\mathcal{T}_h} - \frac{1}{\delta+2}\left(u^{\delta}\sum_{i=1}^d\frac{\partial w}{ \partial x_i},u\right)_{\mathcal{T}_h}.
\end{align} 
This kind of construction is useful as $b_{CR}(u;u,u) = 0$,
so we can prove the stability without any condition on the parameters, as shown in Lemma \ref{2.stabilitysd}. 
Note that
\begin{align}\label{2.7}
	(c(u),u)
	&\leq (1+\gamma)\|u\|^{\delta+1}_{\L^{2(\delta+1)}}\|u\|_{\L^2}-\gamma\|u\|_{\L^2}^2-\|u\|_{\L^{2(\delta+1)}}^{2(\delta+1)}, \qquad \forall ~u\in\L^{2(\delta+1)}(\Omega).
\end{align}
 The discrete energy norm for CR approximation is defined as $|\!|\!|{v}|\!|\!|^2_{CR}:= \int_0^T\|\nabla_h u (s)\|_{\L^2(\mathcal{T}_h)}^2\mathrm{~d}s$.

 The stability estimate for the semi-discrete system defined in \eqref{2.ncweakform} is discussed in the following lemma.
\begin{lemma}\label{2.stabilitysd}
	Assume that $f\in\L^2(0,T;\L^2(\Omega))$ and $u_0\in \L^2(\Omega)$. The weak solution $u_h\in V_h$ of the semi-discrete formulation \eqref{2.ncweakform} satisfies the following stability estimate:
	\begin{align}\label{2.13} 
		&\sup_{0\leq t\leq T}\|u_h(t)\|_{\L^2}^2+\nu|\!|\!|{u_h}|\!|\!|^2_{CR}\leq\left(\|u_0\|_{\L^2}^2+\frac{1}{\nu}\int_0^T\|f(t)\|_{\L^2}^2\mathrm{~d}t\right)e^{\beta(1+\gamma^2)T}.
	\end{align}
\end{lemma}
\begin{proof}
	 Choosing $\chi= u_h$ in \eqref{2.ncweakform}, and using $b_{CR}(u,u,u) = 0,$ with the estimate \eqref{2.7}, we have 
	\begin{align*}
		&\frac{1}{2}\frac{\mathrm{~d}}{\mathrm{~d}t}\|u_h(t)\|_{\L^2}^2+\nu\|\nabla_h u_h(t)\|_{\L^2(\mathcal{T}_h)}^2+\beta\gamma\|u_h(t)\|_{\L^2}^2+\beta\|u_h(t)\|_{\L^{2(\delta+1)}}^{2(\delta+1)}\nonumber\\&\quad+\eta((K*\nabla_h u_h)(t),\nabla_h u_h(t))=\beta(1+\gamma)(u_h^{\delta+1}(t),u_h(t)) + (f(t),u_h(t)).
	\end{align*}
	for a.e. $t\in[0,T]$. Using Cauchy-S\'chwarz, Poincar\'e and Young's inequality, we find that
	\begin{align*}
		&\nonumber\frac{1}{2}\frac{\mathrm{~d}}{\mathrm{~d}t}\|u_h(t)\|_{\L^2}^2+ \frac{\nu}{2}\|\nabla_h u_h(t)\|_{\L^2(\mathcal{T}_h)}^2+\beta\gamma\|u_h(t)\|_{\L^2}^2+	\frac{\beta}{2}\|u_h(t)\|_{\L^{2(\delta+1)}}^{2(\delta+1)}\\&+\eta((K*\nabla_h u_h)(t),\nabla_h u_h(t))\leq \frac{\beta(1+\gamma)^2}{2}\|u_h(t)\|_{\L^2}^2 +\frac{1}{\nu}\|f\|^2_{\L^2}.
	\end{align*}
	Integrating w.r.t. time, we get
	\begin{align}\label{2.11}
		&\|u_h(t)\|_{\L^2}^2+\nu|\!|\!|u_h|\!|\!|^2_{CR} +\beta\int_0^t\|u_h(s)\|_{\L^{2(\delta+1)}}^{2(\delta+1)}\mathrm{~d}s+2\eta\int_0^t((K*\nabla_h u_h)(s),\nabla_h u_h(s))\mathrm{~d}s\nonumber\\&\leq\|u_0\|_{\L^2}^2+\frac{1}{\nu}\int_0^t\|f(s)\|_{\L^2}^2\mathrm{~d}s+\beta(1+\gamma^2)\int_0^t\|u_h(s)\|_{\L^2}^2\mathrm{~d}s, \quad \forall ~ t\in[0,T].
	\end{align}
	 As the kernel $K(\cdot)$ is a positive kernel \eqref{2.pk}, and using Gronwall's inequality in \eqref{2.11} yields 
	\begin{align}
\nonumber	\|u_h(t)\|_{\L^2}^2+\nu|\!|\!|{u_h}|\!|\!|^2_{CR}+\beta\int_0^t\|u_h(s)\|_{\L^{2(\delta+1)}}^{2(\delta+1)}\mathrm{~d}s\leq\left(\|u_0\|_{\L^2}^2+\frac{1}{\nu}\int_0^T\|f(t)\|_{\L^2}^2\mathrm{~d}t\right)e^{\beta(1+\gamma^2)T}, 
	\end{align}
 $\forall~ t\in[0,T]$. Notably, the RHS is independent of $h.$ Taking supreme over time $0\leq t\leq T,$ leads to the stated result. 
\end{proof}
\begin{lemma}\label{2.crnclem111}
	There holds:
	\begin{align*}
		-\alpha[b_{CR}(u_h;u_h,w)-b_{CR}(v_h;v_h,w)]&\le  \frac{\nu}{2}\|\nabla_h w\|_{\L^2(\mathcal{T}_h)}^2+C(\alpha,\nu)\left(\|u_h\|^{\frac{8\delta}{4-d}}_{\L^{4\delta}}+\|v_h\|^{\frac{8\delta}{4-d}}_{\L^{4\delta}}\right)\|w\|_{\L^2}^2,\\
		A_{CR}(u_h,w)-A_{CR}(v_h,w)&\ge \frac{\nu}{2}\|\nabla_h w\|_{\L^2(\mathcal{T}_h)}^2 +\frac{\beta}{4}(\|{u}^{\delta}_hw\|_{\L^2}^2+\|v^{\delta}_hw\|_{\L^2}^2)\nonumber\\
		&\quad+\left(\beta\gamma-C(\beta,\alpha,\delta) - C(\alpha,\nu)\Big(\|u_h\|^{\frac{8\delta}{4-d}}_{\L^{4\delta}}+\|v_h\|^{\frac{8\delta}{4-d}}_{\L^{4\delta}}\Big)\right)\|w\|_{\L^2}^2,
	\end{align*}
	where $u_h,v_h\in V_{h}$, $w=u_h-v_h$,   $C(\alpha, \nu) = \left(\frac{4+d}{4\nu}\right)^{\frac{4+d}{4-d}}\left(\frac{4-d}{8}\right)(\frac{2^{\delta-1}C\alpha}{(\delta+2)(\delta+1)})^{\frac{4-d}{8}}$ and $\mathcal{C}(\beta,\gamma,\delta)= \frac{\beta}{2}2^{2\delta}(1+\gamma)^2(\delta+1)^2$ is a positive constant depending on parameters.
\end{lemma}
\begin{proof}
To prove the first bound, we use Cauchy-S\'chwarz, inverse inequality, Taylor’s formula, H\"older's and Young’s inequalities such that 
	\begin{align}\label{2.30}
			\nonumber&-\alpha[b_{CR}(u_h;u_h,w)-b_{CR}(v_h;v_h,w)]
			\\&=\frac{-\alpha}{\delta+2}\sum_{K\in \mathcal{T}_h}\sum_{i=1}^d  \left(\int_K\left(u_h^{\delta}\frac{\partial u_h}{\partial_{x_i}}-v_h^{\delta}\frac{\partial v_h}{\partial_{x_i}}\right) w \mathrm{d}x- \int_K(u_h^{\delta+1}-v_h^{\delta+1})\frac{\partial w}{\partial{x_i}} \mathrm{d}x\right)\nonumber\\
			&\leq \frac{\nu}{2}\|\nabla_h w\|_{\L^2(\mathcal{T}_h)}^2+C(\alpha,\nu)\left(\|u_h\|^{\frac{8\delta}{4-d}}_{\L^{4\delta}}+\|v_h\|^{\frac{8\delta}{4-d}}_{\L^{4\delta}}\right)\|w\|_{\L^2}^2,
	\end{align}
where $C(\alpha, \nu) = \left(\frac{4+d}{4\nu}\right)^{\frac{4+d}{4-d}}\left(\frac{4-d}{8}\right)(\frac{2^{\delta-1}C\alpha}{(\delta+2)(\delta+1)})^{\frac{8}{4-d}}$. Now, we estimate the term $\beta(c(u_h)-c(v_h),w)$ as  
\begin{align}\label{2.24}
	\beta(c(u_h)-c(v_h),w)
	=-\beta\gamma\|w\|_{\L^2}^2-\beta(u_h^{2\delta+1}-v_h^{2\delta+1},w)+\beta(1+\gamma)(u_h^{\delta+1}-v_h^{\delta+1},w). 
\end{align}
Using (2.21) and (2.22) of \cite{GBHE},
gives 
\begin{align}\label{2.27}
	\beta&\left[(u_h(1-u_h^{\delta})(u_h^{\delta}-\gamma)-v_h(1-v_h^{\delta})(v_h^{\delta}-\gamma),w)\right] \nonumber\\
	&\qquad\qquad\leq -\beta\gamma\|w\|_{\L^2}^2-\frac{\beta}{4}\|u_h^{\delta}w\|_{\L^2}^2-\frac{\beta}{4}\|v_h^{\delta}w\|_{\L^2}^2+\frac{\beta}{2}2^{2\delta}(1+\gamma)^2(\delta+1)^2\|w\|_{\L^2}^2.
\end{align}
Combining \eqref{2.30}-\eqref{2.27}, we obtain 
\begin{align}
	&\nonumber A_{CR}(u_h,w) - A_{CR}(v_h,w)\\\nonumber&= \nu a_{CR}(u_h-v_h,w) + \alpha (b_{CR}(u_h,u_h,w)-b_{CR}(v_h,v_h,w))-\beta(c(u_h)-c(v_h),w)\\&\nonumber\geq \nu\|\nabla_h w\|_{\L^2(\mathcal{T}_h)}^2-\frac{\nu}{2}\|\nabla_h w\|_{\L^2(\mathcal{T}_h)}^2-C(\alpha,\nu)\left(\|u_h\|^{\frac{8\delta}{4-d}}_{\L^{4\delta}}+\|v_h\|^{\frac{8\delta}{4-d}}_{\L^{4\delta}}\right)\|w\|_{\L^2}^2
	\\
	&\quad+\left(\beta\gamma-\mathcal{C}(\beta,\gamma,\delta) - C(\alpha,\nu)\Big(\|u_h\|^{\frac{8\delta}{4-d}}_{\L^{4\delta}}+\|v_h\|^{\frac{8\delta}{4-d}}_{\L^{4\delta}}\Big)\right)\|w\|_{\L^2}^2,
\end{align}
where $\mathcal{C}(\beta,\gamma,\delta)= \frac{\beta}{2}2^{2\delta}(1+\gamma)^2(\delta+1)^2$, gives the required result.
\end{proof}
Next, we discuss the existence of a unique solution of the semi-discretized system.  
\begin{theorem}
	For  $f \in \L^2(0,T;\L^2(\Omega)), u_0\in \L^2(\Omega)$ there exist at least one solution $u_h\in V_h$. Moreover, for $u_h^0\in\L^{d\delta}(\Omega)$,  the weak solution to the system \eqref{2.ncweakform} is unique. 
\end{theorem}
\begin{proof}
	\noindent\textbf{Step 1:} \emph{Existence.} For the existence of a discrete solution, we will show that the operators defined are Lipschitz and use the results of ODE as done in \cite[Theorem 3.2]{DHY}.
	 For $ u,v,z \in \H^1(\mathcal{T}_h)\cap \L^{2(\delta+1)}(\Omega),$ and $w=u-v$, by employing integration by parts and further using Taylor's formula for $0<\theta<1$, we achieve
	\begin{align}\label{2.bllc}
	\nonumber	\langle B(u)-B(v),z\rangle 
	& =\frac{1}{(\delta+2)}\Bigg[ -\delta\left(u^{\delta-1}\sum_{i=1}^{d}\frac{\partial u}{\partial x_i}w,z\right) -\left(u^{\delta}w,\sum_{i=1}^d\frac{\partial z}{\partial x_i}\right)\\&\quad+ \delta\left((\theta u+(1-\theta)v)^{\delta-1}\sum_{i=1}^{d}\frac{\partial v}{\partial x_i}w,z\right)\nonumber\\&\quad+ (\delta+1)\left((\theta u+(1-\theta)v)^{\delta}w,\sum_{i=1}^{d}\frac{\partial z}{\partial x_i}\right)\Bigg]\nonumber\\
	&\leq C\rho^{\delta}\|u-v\|_{\L^{2(\delta+1)}}\|w\|_{\L^{2(\delta+1)}}\|z\|_{\H^1(\mathcal{T}_h)\cap\L^{2(\delta+1)}},
\end{align}
	$\forall~ \|u\|_{\H^1(\mathcal{T}_h)\cap\L^{2(\delta+1)}}, \|v\|_{\H^1(\mathcal{T}_h)\cap\L^{2(\delta+1)}}\leq \rho$, where $\langle B(u),z\rangle = b_{CR}(u,u,z)$.
	
 Again, assume $u,v,z\in\L^{2(\delta+1)}(\Omega)$ such that $\|u\|_{\L^{2(\delta+1)}},\|v\|_{\L^{2(\delta+1)}}\leq \rho$  and $w = u-v$. For $0<\theta_2<1$ and $0<\theta_3<1$, an application of Taylor's formula and H\"older's inequality yields
	\begin{align}\label{2.cllc}
		\nonumber&\langle c(u)-c(v),z\rangle 
		\\&\leq C\rho \left((1+\gamma)(\delta+1)2^{\delta}|\Omega|^{\frac{\delta}{2(\delta+1)}}\rho^{\delta}+\gamma|\Omega|^{\frac{\delta}{\delta+1}}+(2\delta+1)2^{2\delta}\rho^{2\delta}\right)\|w\|_{\L^{2(\delta+1)}}\|z\|_{\L^{2(\delta+1)}},
	\end{align}
where $|\Omega|$ represents the Lebesgue measure of $\Omega$.
 Using the results discussed in  \cite[Theorem 3.1]{DHY}
with \eqref{2.bllc}-\eqref{2.cllc},  the discrete system \eqref{2.ncweakform} has a local solution. 
		\vskip 2mm
	\noindent\textbf{Step 2:} \emph{Uniqueness.}
For given $f(\cdot,\cdot)$ and $u^0_h(\cdot)$, let the discrete formulation \eqref{2.ncweakform} have two weak solutions, $u^1_h(\cdot)$ and $u^2_h(\cdot)$.
Then, $\omega=u^1_h-u^2_h$ satisfies: 
	\begin{align*}
		&\nonumber(\partial_t(u^1_h-u^2_h), \chi)+A_{CR}(u^1_h,\chi)-A_{CR}(u^2_h,\chi)+\eta\int_0^tK(t-\tau)a_{CR}((u^1_h-u^2_h)(\tau),\chi) ~\mathrm{d}\tau = 0,
	\end{align*}
	for $(x,t)\in\Omega\times(0,T)$. Using $\chi=\omega = u^1_h-u^2_h \in V_h$, we find 
	\begin{align}\label{2.36}
		\frac{1}{2}\frac{\mathrm{~d}}{\mathrm{~d}t}\|\omega\|_{\L^2}^2+A_{CR}(u^1_h,\omega)-A_{CR}(u^2_h,\omega)+\eta( (K*\nabla_h \omega),\nabla_h \omega)= 0.
	\end{align} 
Using Lemma \ref{2.crnclem111}, we have 
	\begin{align*}
		\frac{\mathrm{~d}}{\mathrm{~d}t}\|\omega(t)\|_{\L^2}^2+&\nu\|\nabla_h \omega(t)\|_{\L^2(\mathcal{T}_h)}^2+\frac{\beta}{2}\left(\|u^1_h(t)^{\delta}\omega(t)\|_{\L^2}^2 + \|u^2_h(t)^{\delta}\omega(t)\|_{\L^2}^2\right)+ \beta\gamma\|\omega(t)\|_{\L^2}^2\nonumber\\&
		+\eta( (K*\nabla_h \omega)(t),\nabla_h \omega(t))
		\leq \left(C(\beta,\alpha,\delta) + C(\alpha,\nu)\Big(\|u^1_h\|^{\frac{8\delta}{4-d}}_{\L^{4\delta}}+\|u^2_h\|^{\frac{8\delta}{4-d}}_{\L^{4\delta}}\Big)\right)\|\omega\|_{\L^2}^2. 
	\end{align*}
As a result of integrating the above inequality, ensuring the positivity of the kernel $K$, and subsequently applying Gronwall's inequality, we find:
	\begin{align*}
	\|\omega(t)\|_{\L^2}^2\leq\|\omega(0)\|_{\L^2}^2e^{ \beta 2^{2\delta}(1+\gamma)^2(\delta+1)^2T}\exp\left\{C(\alpha,\nu)\int_0^T\left(\|u^1_h(t)\|^{\frac{8\delta}{4-d}}_{\L^{4\delta}}+\|u^2_h(t)\|^{\frac{8\delta}{4-d}}_{\L^{4\delta}}\right)\mathrm{~d} t\right\} ,
\end{align*}
  $\forall~t\in[0,T]$. For $u^1_h(0)\in \L^{d\delta}(\Omega),$ the term in exponential is bounded. As $\omega(0)=0$ and $u^1_h$ and $u^2_h$ satisfies the system (\ref{2.ncweakform}), uniqueness follows easily.
\end{proof}
Subsequently,  we denote the usual finite element interpolation \cite{VJL} by $I_h$, such that
\begin{align}\label{2.ncapprox11}
	|v-I_hv|_{\H^m(K)}&\le Ch^{2-m}_K\|v\|_{\H^2(K)},\quad v\in \H^2(K),\\
\nonumber	\|v-(I_hv)\|_{\L^2(E)}&\le Ch^{3/2}\|v\|_{\H^2(K)},\quad v\in \H^2(K)\quad E\in\mathcal{E}(\mathcal{T}_h). \label{2.ncapprox12}
\end{align}
Concerning the edge projection operator denoted as  $P_E:L^2(E)\rightarrow P_0(E)$, where $P_0(E)$ is a constant on $E$, we have 
\begin{align}\label{2.P0projection}
	\|v-P_E v\|_{\L^2(E)}\le Ch^{1/2}_K|v|_{\H^1(K)},\forall \  v\in \H^1(K),\ E\in\mathcal{E}(\mathcal{T}_h).
\end{align}
\begin{theorem}\label{2.SDNCFEM}
	Assume that $u$  and $u_h$ be the weak solutions of \eqref{2.weakform} and \eqref{2.ncweakform} on the interval $(0,T]$ respectively. If we assume initial data $u_0\in X_{d\delta}$ and the forcing  $f\in\mathrm{L}^2(0,T;\L^2(\Omega))$, then the semi-discrete solution $u_h$ of the NCFEM tends to the exact solution $u$ as $h\rightarrow0$. Additionally, the following assertion  holds
	\begin{align*}
		\|u_h-u\|_{\L^{\infty}(0,T;\L^2(\Omega))}^2 +	|\!|\!|u_h-u|\!|\!|_{CR}^2 \leq C\bigg\{\|u^h_0-u_0\|_{\L^2}^2 + h^2~\Theta(u)\bigg\},
\end{align*}
	where the constant $C$ depends on parameters $\alpha,\beta,\gamma,\delta,$ but independent of $h$ and 
	$$\Theta(u) =  \int_0^T\|u(t)\|_{\H_0^1}^2\mathrm{~d} t +\int_0^T\|u(t)\|_{\H^2}^2\mathrm{~d} t+\int_0^T\|\partial_tu(t)\|^2_{\H_0^1}\mathrm{~d} t.\label{2.th}$$
\end{theorem}
\begin{proof} 
	Applying triangle inequality gives
	$$|\!|\!|u_h-u|\!|\!|_{CR}   \leq |\!|\!|u_h -W|\!|\!|_{CR} + |\!|\!|W-u|\!|\!|_{CR}.$$
	Now, for the second term we have  $|\!|\!|W-u|\!|\!|_{CR} \leq C h,$ using $\H^1$projection (Ritz-Projection). So, our aim is to estimate $|\!|\!|u_h -W|\!|\!|_{CR} $.

	Using regularity result of Theorem \ref{2.weaksolution}, it holds
	\begin{align}\label{2.ncexact}
		\nonumber&(\partial_tu, \chi) +\nu a_{CR}(u,\chi) + \alpha b_{CR}(u,u,\chi) -\beta (c(u),\chi)+\eta\int_0^t K(t-\tau) a_{CR}(u(\tau),\chi) ~\mathrm{d}\tau \\&=(f,\chi)+\sum_{K\in\mathcal{T}_h}\int_{\partial K}\nu\frac{\partial u}{\partial n_K} \chi ~\mathrm{d}s +\sum_{K\in\mathcal{T}_h}\int_{\partial K}\eta  \left(K*\frac{\partial u}{\partial n_K}\right) \chi ~\mathrm{d}s - \frac{\beta}{\delta +2}\sum_{K\in\mathcal{T}_h}\int_{\partial K} u^{\delta+1}n^i\chi~\mathrm{d}s  , 
	\end{align}
	$\forall ~\chi\in V_h,$ where $n^i= (n_1, \ldots, n_d),$ denotes the outward unit normal vector.
	From \eqref{2.ncweakform} and \eqref{2.ncexact}, we have 
	\begin{align*}
		& (\partial_t(u_h(t)-u(t)), \chi) + A_{CR}(u_h(t),\chi) - A_{CR}(u(t),\chi)+\eta\int_0^t K(t -\tau)a_{CR}( u_h(\tau), \chi) \mathrm{~d}\tau \\&-\eta\int_0^t K(t -\tau)a_{CR}( u(\tau), \chi)\mathrm{~d}\tau \\&= - \sum_{K\in\mathcal{T}_h}\int_{\partial K}\nu\frac{\partial u}{\partial n_K} \chi  \mathrm{~d}s - \sum_{K\in\mathcal{T}_h}\int_{\partial K}\eta \left(K*\frac{\partial u}{\partial n_K}\right) \chi\mathrm{~d}s + \frac{\beta}{\delta +2}\sum_{K\in\mathcal{T}_h}\int_{\partial K} u^{\delta+1}n^i\chi~\mathrm{d}s.
	\end{align*}
	Let us choose $\chi = u_h - W $ and write  $u_h - u  = u_h- W +W -u $, where $\chi\in V_h$,
	\begin{align*}
		& \frac{1}{2}\frac{d}{dt}\|u_h(t)-W(t)\|^2_{\L^2}+ A_{CR}(u_h(t),u_h(t)-W(t))- A_{CR}(W(t),u_h(t)-W(t)) \\&+\eta\int_0^t K(t -\tau)a_{CR}( u_h(\tau)-W(\tau), u_h(t)-W(t)) \mathrm{~d}\tau \\& = -(\partial_t(W(t)-u(t)), \chi(t)) - (A_{CR}(W(t),\chi(t)) - A_{CR}(u(t),\chi(t)))
		\\&\quad- \eta\int_0^t K(t -\tau)a_{CR}( W(\tau)-u(\tau), \chi(t))\mathrm{~d}\tau  - \sum_{K\in\mathcal{T}_h}\int_{K}\nu\frac{\partial u}{\partial n_K} \chi \mathrm{~d}s \\&\quad - \sum_{K\in\mathcal{T}_h}\int_{K}\eta \left(K*\frac{\partial u(s)}{\partial n_K}\right) \chi(s)\mathrm{~d}s + \frac{\beta}{\delta +2}\sum_{K\in\mathcal{T}_h}\int_{\partial K} u^{\delta+1}(s)n^i\chi(s)~\mathrm{d}s.
	\end{align*}
	Using Lemma \ref{2.crnclem111}  for $w = u_h-W$, we have 
	\begin{align*}
		&\nonumber\frac{d}{dt}\|u_h(t)-W(t)\|^2_{\L^2} +2\nu\|\nabla_h (u_h(t)-W(t))\|_{\L^2(\mathcal{T}_h)}^2 +\frac{\beta}{2}(\|{u}^{\delta}_h(u_h-W)\|_{\L^2}^2+\|W^{\delta}(u_h-W)\|_{\L^2}^2)\\&\nonumber+2\eta\int_0^t K(t -\tau)a_{CR}( u_h(\tau)-W(\tau), u_h(t)-W(t)) \mathrm{~d}\tau\\
		& +2 \left(\beta\gamma-C(\beta,\alpha,\delta) - C(\alpha,\nu)\Big(\|u_h(t)\|^{\frac{8\delta}{4-d}}_{\L^{4\delta}}+\|W(t)\|^{\frac{8\delta}{4-d}}_{\L^{4\delta}}\Big)\right)\|u_h(t)-W(t)\|_{\L^2}^2\nonumber\\
		&\nonumber\leq -2\partial_t(W(t)-u(t),\chi(t)) 
		+\sum_{i=1}^4J_i \nonumber - 2\eta\int_0^t K(t -\tau)a_{CR}( W(\tau)-u_h(\tau), \chi(t)) \mathrm{d}\tau \\
&\quad-2 \int_{\partial K}\nu\frac{\partial u(s)}{\partial n_K} \chi(s)  \mathrm{~d}s-2 \sum_{K\in\mathcal{T}_h}\int_{\partial K}\eta \left(K*\frac{\partial u}{\partial n_K}\right)(s) \chi(s) \mathrm{~d}s \\&\quad+ \frac{\beta}{\delta +2}\sum_{K\in\mathcal{T}_h}\int_{\partial K} u^{\delta+1}(s)n^i\chi(s)~\mathrm{d}s,
	\end{align*}
	with 
	\begin{align*}
		J_1 &= (W(t)-u(t),\partial_t(u_h(t)-W(t))), \quad J_2= a_{CR}( W(\tau)-u(\tau), u_h(t)-W(t)),\\
		J_3 &= -\frac{2\alpha}{\delta+2}\sum_{K\in \mathcal{T}_h}\sum_{i=1}^d  \left(\int_K\left(W^{\delta}\frac{\partial W}{\partial_{x_i}}-u^{\delta}\frac{\partial u}{\partial_{x_i}}\right) (u_h - W) \mathrm{~d}x- \int_K(W^{\delta+1}-u^{\delta+1})\frac{\partial (u_h-W)}{\partial{x_i}} \mathrm{d}x\right), \\
		J_4 &= \beta\left[(W(1-W^{\delta})(W^{\delta}-\gamma)-u(1-u^{\delta})(u^{\delta}-\gamma),u_h-W)\right].
	\end{align*}
	Using \cite[Theorem 10.3.11]{SCB}, it follows: 
	\begin{align}
		\sum_{K\in\mathcal{T}_h}\int_{\partial K}\nu\frac{\partial u}{\partial n_K} \chi=\sum_{E\in\mathcal{E}}\int_{E}\nu\frac{\partial u}{\partial n_E} [\chi] 
		=\sum_{E\in\mathcal{E}}\int_{E}\nu\left(\frac{\partial u}{\partial n_E} -P\left(\frac{\partial u}{\partial n_E}\right)\right)[\chi].
	\end{align}
	Therefore, we can utilize the estimate (\ref{2.P0projection}), which yields 
	\begin{align*}
		\nonumber\left|\sum_{K\in\mathcal{T}_h}\int_{\partial K}\nu\frac{\partial u}{\partial n_K} \chi\right| &\le C \left(\sum_{K\in\mathcal{T}_h}\nu h_K^2\|u\|_{\H^2(K)}^2\right)^{1/2}\|\nabla_h \chi\|_{\L^2(\mathcal{T}_h)}\\&\le C \sum_{K\in\mathcal{T}_h}\nu h_K^2\|u\|_{\H^2(K)}^2 +\frac{\nu}{4} \|\nabla_h(u_h-W) \|^2_{\L^2(\mathcal{T}_h)},
	\end{align*}
	Again, using  \cite[Theorem 10.3.11]{SCB} and the Bramble-Hilbert lemma, we have 
	\begin{align*}
		&\frac{\beta}{\delta +2}\sum_{K\in\mathcal{T}_h}\int_{\partial K} u^{\delta+1}n^i \chi
		\leq C^2 \sum_{K\in\mathcal{T}_h}\nu h_K^2\|u^{\delta+1}\|_{\H^1(K)}^2 +\frac{\nu}{4} \|\nabla_h(u_h-W) \|^2_{\L^2(\mathcal{T}_h)},
	\end{align*}
Moreover, $J_1$ and  $J_2$ satisfies the following bound:
	\begin{align*}
		|J_1|&\leq
\|W-u\|_{\mathrm{L}^2}^2+\|\partial_t(u_h-W)\|_{\mathrm{L}^2}^2,\\|J_2|&
\leq\frac{4}{\nu}\|\nabla_h(W-u)\|_{\L^2(\mathcal{T}_h)}^2+\frac{\nu}{4}\|\nabla_h(u_h-W)\|_{\L^2(\mathcal{T}_h)}^2.
	\end{align*}
	To estimate $J_3,$ we first apply an integration by parts with the inverse inequality. Further, employing Taylor's formula with H\"older's and Young's inequalities yields
	\begin{align*}
		\nonumber	|J_3|
		&= -\frac{2\alpha}{\delta+2}\sum_{K\in \mathcal{T}_h}\sum_{i=1}^d  \left(\int_K\left(W^{\delta}\frac{\partial W}{\partial{x_i}}-u^{\delta}\frac{\partial u}{\partial{x_i}}\right) (u_h-W) \mathrm{~d}x- \int_K(W^{\delta+1}-u^{\delta+1})\frac{\partial (u_h-W)}{\partial{x_i}} \mathrm{d}x\right) 
		\\& \leq\frac{2^{2\delta}\alpha^2\delta^2}{\nu(\delta+2)^2}\left(\|{W}\|_{\L^{2(\delta+1)}}^{2\delta}+\|{u}\|_{\L^{2(\delta+1)}}^{2\delta}\right)\|\nabla_h(W-u)\|_{\L^2(\mathcal{T}_h)}^2+\frac{\nu}{2}\|\nabla_h(u_h-W)\|_{\L^2(\mathcal{T}_h)}^2.
	\end{align*}
	Let us rewrite $J_4$ as	$J_4= J_5 + J_6 + J_7,$ where 
\begin{align*}
	J_5 &= 2\beta(1+\gamma)({W}^{\delta+1}-u^{\delta+1},u_h-W),\quad  J_6 = -2\beta\gamma(W-u,u_h-W)\\ J_7 &= -2\beta({W}^{2\delta+1}-u^{2\delta+1},u_h-W).
\end{align*}
		The term $J_5$ can be estimated first using Taylor's formula, then H\"older's and finally Young's inequalities as
	\begin{align}
		|J_5|&=2\beta(1+\gamma)(\delta+1)((\theta W+(1-\theta)u)^{\delta}(W-u),u_h-W) 
		\nonumber\\&\leq 2^{2\delta-1}\beta(1+\gamma)^2(\delta+1)^2\left(\|{W}\|^{2\delta}_{\L^{4\delta}}+\|{u}\|^{2\delta}_{\L^{4\delta}}\right)\|u_h-W\|^2_{\L^2} + \frac{\beta}{2}\|\nabla_h(W-u)\|^2_{\L^2(\mathcal{T}_h)}.
	\end{align}
	We estimate $J_6$, by first using Cauchy-S\'chwarz inequality and then Young's inequality as
	\begin{align}
		|J_6|&\leq 2\beta\gamma\|W-u\|_{\L^2}\|u_h-W\|_{\L^2}\leq 2\beta\gamma\|W-u\|_{\L^2}^2+\frac{\beta\gamma}{2}\|u_h-W\|_{\L^2}^2.
	\end{align}
	Making use of Taylor's formula, H\"older, Young's inequality and the discrete Sobolev embedding, we estimate $J_7$ as
	\begin{align}\label{2.7p11}
		|J_7|&=-2(2\delta+1)\beta\left((\theta W+(1-\theta)u)^{2\delta}(W-u),u_h-W\right)
		\nonumber\\&\leq  2^{2\delta}(2\delta+1)\beta\left(\|W\|_{\L^{4\delta}}^{2\delta}+\|u\|_{\L^{4\delta}}^{2\delta}\right)\|W-u\|_{\L^{2d}}\|u_h-W\|_{\L^{\frac{2d}{d-1}}}	
		\nonumber\\&\leq 2^{2\delta-1}(2\delta+1)\beta\|\nabla_h(W-u)\|_{\L^2(\mathcal{T}_h)}^2 + \frac{\nu}{4}\|\nabla_h(u_h-W)\|_{\L^2(\mathcal{T}_h)}^2 
		\nonumber\\&\quad + \frac{2^{4\delta-2}(2\delta+1)^2\beta^2}{\nu}\left(\|W\|_{\L^{4\delta}}^{8\delta}+\|u\|_{\L^{4\delta}}^{8\delta}\right)\|u_h -W\|_{\L^2}^2.
	\end{align}
			Substituting back the above estimate, integrating from $0$ to $t$, using positivity of kernel and the estimates
			\begin{align*}
				&\eta\int_0^t\langle K*\nabla_h(W(s)-u(s)),\nabla_h(u_h(s)-W(s))\rangle\mathrm{~d}s\\
				&\leq \frac{\nu}{4}\int_0^t\|\nabla_h(W(s)-u(s))\|_{\mathrm{L}^2(\mathcal{T}_h)}^2\mathrm{~d}s+\frac{C_K\eta^2}{\nu}\int_0^t\|\nabla_h(u_h(s)-W(s))\|_{\mathrm{L}^2(\mathcal{T}_h)}^2\mathrm{~d}s,
			\end{align*}
			where the constant $C_K = \int_0^T|K(t)|\mathrm{~d}t$, and  
			\begin{align*}
				\left|\sum_{K\in\mathcal{T}_h}\int_0^t\int_{\partial K}\eta \left(K*\frac{\partial u}{\partial n_K}\right) \chi\right| 
				&\le C_{K}\frac{ \eta^2}{\nu} \sum_{K\in\mathcal{T}_h} h_K^2\int_0^t\|u\|_{\H^2(K)}^2 + \frac{\nu}{4}\int_0^t \|\nabla_h(u_h-W) \|^2_{\L^2(\mathcal{T}_h)},
			\end{align*} we obtain	 
			\begin{align}\label{2.715}
	&\|u_h(t)-W(t)\|_{\L^2}^2+\nu \int_0^t\|\nabla_h(u_h(s)-W(s))\|_{\L^2(\mathcal{T}_h)}^2\mathrm{~d}s+\int_0^t\|u_h(s)-W(s))\|_{\L^{2\delta+2}}^{2\delta+2}\mathrm{~d}s\nonumber
	\\&\nonumber\leq \|u_h^0-W(0)\|_{\L^2}^2-2(W(t)-u(t),u_h(t)-W(t))+2(W(0)-u_0,u_h(0)-W(0))	\\&\nonumber\quad+\int_0^t\|\partial_t(u_h(s)-W(s))\|_{\mathrm{L}^2}^2\mathrm{~d}s \nonumber +\int_0^t \bigg(\frac{4}{\nu}+\frac{2^{2\delta}\alpha^2\delta^2}{\nu(\delta+2)^2}\left(\|{W(s)}\|_{\L^{2(\delta+1)}}^{2\delta}+\|{u(s)}\|_{\L^{2(\delta+1)}}^{2\delta}\right)\\&\nonumber\quad + \frac{\beta}{2}+ 2^{2\delta-1}(2\delta+1)\beta\bigg)\|\nabla_h (W(s)-u(s))\|_{\L^2(\mathcal{T}_h)}^2\mathrm{~d}s
	 +\bigg(1+2\beta\gamma\bigg)\int_0^t\|W(s)-u(s)\|_{\mathrm{L}^2}^2 \mathrm{~d}s\\&\nonumber\quad  + C^2 \sum_{K\in\mathcal{T}_h}\nu h_K^2\int_0^t\|u^{\delta+1}\|_{\H^1(K)}^2 + \left(C^2 + C_{K}\frac{ \eta^2}{\nu} \right) \sum_{K\in\mathcal{T}_h} \nu h_K^2\int_0^t\|u\|_{\H^2(K)}^2\nonumber
	\\&\quad\nonumber+\int_0^t\bigg(2^{2\delta-1}\beta(1+\gamma)^2(\delta+1)^2\left(\|{W}\|^{2\delta}_{\L^{4\delta}}+\|{u}\|^{2\delta}_{\L^{4\delta}}\right) + \frac{\beta\gamma}{2} + C(\alpha,\nu)\Big(\|u(s)\|^{\frac{8\delta}{4-d}}_{\L^{4\delta}}+\|W(s)\|^{\frac{8\delta}{4-d}}_{\L^{4\delta}}\Big) \\&\quad+ C(\beta,\alpha,\delta) \frac{2^{4\delta-2}(2\delta+1)^2\beta^2}{\nu}\left(\|W(s)\|_{\L^{4\delta}}^{8\delta}+\|u(s)\|_{\L^{4\delta}}^{8\delta}\right)   \bigg) \|u_h(s)-W(s)\|^2_{\L^2}\mathrm{~d}s,  \qquad \forall ~t\in[0,T].
\end{align}	
		Using Cauchy-S\'chwarz and AM-GM inequality it follows that
			\begin{align*}
				-2(W-u,u_h-W)&
\leq\frac{1}{2}\|W-u\|_{\L^2}^2+2\|u_h-W\|_{\L^2}^2,\\
				2(W(0)-u_0,u_h(0)-W(0))&
\leq\|W(0)-u_0\|_{\L^2}^2+\|u_h(0)-W(0)\|_{\L^2}^2.
			\end{align*}
			Substituting back in \eqref{2.715}, applying Gronwall's inequality and the bounds for interpolation  \eqref{2.ncapprox11} leads to the stated result.
		\end{proof}
\subsubsection{Fully-discrete non-conforming FEM}\label{sec2.2.2}
This section deals with the fully-discrete finite element scheme in both space and time. The time interval $[0,T]$ is partitioned  into, $0 =t_0<t_1<t_2,\cdots < t_N = T$ with the uniform time stepping $\Delta t$. Then, we the apply backward Euler method to discretize  the time derivative. Moreover, the memory term is approximated by the positive implicit quadrature rule as:
\begin{align*}
	J(\psi) = \int_0^{t} K(t-s)\psi(s) \mathrm{d}s \approx	\frac{1}{(\Delta t)^2} \int_{t_{k-1}}^{t_k}\int_0^t K(t-s)\Delta t \psi(s)\mathrm{~d}s\mathrm{d}t  =  \sum_{j=1}^k \omega_{kj}\Delta t \psi^j,
\end{align*}
where $\omega_{k j}=\frac{1}{(\Delta t)^2 } \int_{t_{k-1}}^{t_k} \int_{t_{j-1}}^{\min \left(t, t_j\right)} K(t-s) \mathrm{~d}s \mathrm{~d}t$, for $1\leq k\leq N$ and $\psi^j = \psi(t_j)$ in $(t_{j-1},t_{j})$.
The fully-discrete weak formulation of the system (\ref{2.GBHE}) reads as: Given $u_h^{k-1}$, find $ u_h^k\in V_h$ such that
	\begin{align}\label{2.ncweakformfd}
	\nonumber	(\bar{\partial}u_h^k, \chi)+ A_{CR}(u_h^k,\chi)+\eta \left(\sum_{j=1}^{k}\omega_{kj}\Delta t\nabla_h u_h^{j}, \nabla_h \chi \right)= (f^k,\chi),\\
		(u_h(0),\chi)=(u_h^0,\chi), 
	\end{align}
for $\chi\in V_h$, where,  $u^0_h$ is the approximation of $u_0$ in $V_h$, 
$f^k = (\Delta t)^{-1}\int_{t_{k-1}}^{t_k} f(s)  \mathrm{~d}s,$ for $f\in \L^2(0,T;\L^2(\Omega))$, 
\[\bar{\partial}u_h^k= \frac{u_h^k-u_h^{k-1}}{\Delta t},\quad a_{CR}(u,v) = (\nabla_h u, \nabla_h v), \]
and the associated discrete energy norm is defined as,  
$|\!|\!|{v^k}|\!|\!|_{CR} 
	 := \Delta t \sum\limits_{k=1}^N\|\nabla_h v^k\|_{\L^2(\mathcal{T}_h)}^2$.
 We then define the fully-discrete finite element approximation solution for $t\in [t_{k-1},t_k]$  by
\begin{align}\label{2.gfds}
	u_{kh} |_{[t_{k-1},t_k]} = u_h^{k-1}+\left(\frac{t-t_{k-1}}{\Delta t}\right)(u_h^{k}-u_h^{k-1}), \qquad 1\leq k\leq N.
\end{align}

\begin{lemma}\label{2.lemma3.1}
	Let us define the set  $\{f^k\}_{k=1}^N$ by $f^k = (\Delta t)^{-1}\int_{t_{k-1}}^{t_k} f(s)  \mathrm{d}s$.
	If $f\in \L^2(0,T;\L^{2}(\Omega)),$ then we have
	\begin{align*}
		\Delta t \sum_{k=1}^N\|f^k\|_{\L^{2}}^2 \leq C \|f\|^2_{\L^2(0,T;\L^{2}(\Omega))}\quad\mbox{and}\quad
		\sum_{k=1}^N\int_{t_{k-1}}^{t_k}\|f^k-f(t)\|_{\L^{2}}^2 \mathrm{d}t \rightarrow 0 \quad \Delta t \rightarrow 0.
	\end{align*}
	Further if $f\in H^{\epsilon}(0,T;\L^{2}(\Omega))$ for some $\epsilon\in [0,1]$, then 
	\begin{align*}
		\sum_{k=1}^N\int_{t_{k-1}}^{t_k}\|f^k-f(t)\|_{\L^{2}}^2\mathrm{d}t \leq C (\Delta t)^{2\epsilon}\|f\|^2_{\H^{\epsilon}(0,T;\L^{2}(\Omega))}.
	\end{align*}
\end{lemma}
\begin{proof}
The above result has been proven in  \cite[Lemma 3.2]{LSW}.
\end{proof}
 The stability estimate for the fully-discrete approximation \eqref{2.ncweakformfd} is given as
\begin{lemma}[Stability]\label{2.4.1}
	 Let $\{u_h^k\}_{k=1}^N \subset V_h$ be defined by \eqref{2.ncweakformfd}. Assume that, $u_h^0 \in V_h$ and the forcing $f\in \L^2(0,T;\L^{2}(\Omega)),$ then we have  
	\[ |\!|\!|u_h^k|\!|\!|_{CR}\leq C(f,u_0) \times e^{2T\beta(1+\gamma)^2}.\]
\end{lemma}  
\begin{proof}
		Taking $\chi = u_h^k$ in \eqref{2.ncweakformfd}, for $f\in \L^2(0,T;\L^2(\Omega))$,  we achieve
	\begin{align*}
		&\frac{1}{2\Delta t}\|u_h^k\|_{\L^2}^2 - \frac{1}{2\Delta t}\|u_h^{k-1}\|_{\L^2}^2 + \frac{1}{2\Delta t}\|u_h^k-u_h^{k-1}\|_{\L^2}^2+\nu a_{CR}(u_h^k,u_h^k)+ \alpha b_{CR}(u_h^k,u_h^k,u_h^k) \\&+\eta \left(\sum_{j=1}^{k}\kappa_{k-j}\nabla_h u_h^{j} ,\nabla_h u_h^k\right)=\beta(u_h^k(1-(u_h^k)^{\delta})((u_h^k)^{\delta}-\gamma),u_h^k)+( f^k,u_h^k).
	\end{align*}
	Using Cauchy Schwarz, Young’s inequality and the estimate \eqref{2.7}, we achieve
	\begin{align*}
		&\frac{1}{2\Delta t}\|u_h^k\|_{\L^2}^2 - \frac{1}{2\Delta t}\|u_h^{k-1}\|_{\L^2}^2 + \frac{1}{2\Delta t}\|u_h^k-u_h^{k-1}\|_{\L^2}^2+\nu \|\nabla_h u_h^k\|_{\L^2(\mathcal{T}_h)}^2+ \beta\gamma\|u_h^k\|_{\L^2}^2+\beta\|u_h^k\|_{\L^{2(\delta+1)}}^{2(\delta+1)}\\&+\eta \left(\sum_{j=1}^{k}\kappa_{k-j}\nabla_h u_h^{j} ,\nabla_h u_h^k\right)\leq  \frac{\beta}{2}\| u_h^k\|^{2(\delta+1)}_{\L^{2(\delta+1)}}+\frac{\beta(1+\gamma)^2}{2}\|u_h^k\|_{\L^2}^2+ \frac{C_{\Omega}}{\nu}\|f^k\|_{\L^2}^2 + \frac{\nu}{4} \|\nabla u_h^k\|_{\L^2}^2.
	\end{align*}
	where $C_{\Omega}$ is a constant depending on the domain $\Omega$. Summing over $k,$ $k = 1, 2, \cdots, N$, we have 
	\begin{align*}
		&\frac{1}{2\Delta t}\|u_h^N\|_{\L^2}^2+\frac{\nu}{2} \sum_{k=1}^N\|\nabla_h u_h^k\|_{\L^2(\mathcal{T}_h)}^2\leq  \frac{1}{2\Delta t}\|u_h^0\|_{\L^2}^2 + \frac{\beta(1+\gamma)^2}{2}\sum_{k=1}^N\|u_h^k\|_{\L^2}^2 + \frac{C_{\Omega}}{2\nu\Delta t }\Delta t\sum_{k=1}^N\|f^k\|_{\L^2}^2. 
	\end{align*}
where we have used the positivity of the Kernel $K(t)$  \cite[Lemma 4.7]{WVt}. 
Finally, using discrete Gronwall inequality \cite[Lemma 9]{AKP}, we obtain the required bound.
\end{proof}
\begin{lemma}\label{2.thm7.2} Let $\delta \in [1,\infty)$, for $d=2,$ and $\delta\in[1,2]$ for $d=3$. If
 $u_0\in \H^2(\Omega)\cap \H_0^1(\Omega)$ and  $f\in \H^1(0,T;\L^2(\Omega))$,
then the following assertion holds:
	\begin{align*}
	\nonumber\|u_h - u_{kh}\|_{\L^{\infty}(0,T;\L^2(\Omega))}^2 +|\!|\!|u_h - u_{kh}|\!|\!|_{CR}^2&\leq C(f,u_0)\Big((\Delta t)^2+ \eta^2(\Delta t)^2\sup_{k,j} \omega_{kj}^2\Big).
\end{align*}
\end{lemma}
\begin{proof} To obtain the desired result, we first estimate the error at nodal values in Step 1.
	\vskip 2mm
	\noindent\textbf{Step 1:}	For each $t\in[t_{k-1},t_k],$
	integrating the scheme  \eqref{2.ncweakform}, we attain
	\begin{align}\label{2.sdi}
		&\nonumber (u_h({t_{k}})-u_h(t_{k-1}),\chi)+\nu \left(\int_{t_{k-1}}^{t_k}\nabla_h u_h(t) \mathrm{~d} t,\nabla_h \chi\right)+\alpha \left(\int_{t_{k-1}}^{t_k} B_{CR}(u_h(t)) \mathrm{~d} t,\chi
		\right)\\&\qquad+\eta\left(\int_{t_{k-1}}^{t_k}(K*\nabla_h u_h)(t)\mathrm{~d} t,\nabla_h \chi\right) = \beta\left(\int_{t_{k-1}}^{t_k}c(u_h(t))\mathrm{~d} t,\chi\right)+\left( \int_{t_{k-1}}^{t_k}f(t)\mathrm{~d} t,\chi\right),
	\end{align}
	The fully-discrete scheme \eqref{2.ncweakformfd} at $t = t_k,$ is given as
	\begin{align}\label{2.fd3}
		&\nonumber	\left(\frac{u_h^k-u_h^{k-1}}{\Delta t},\chi\right)+\nu a(u_h^k,\chi)+ \alpha b_{CR}(u_h^k,u_h^k,\chi)+\Big(\sum_{j=1}^k \omega_{kj} \Delta t \nabla_h u_{h}^k,\nabla_h \chi\Big) \\&\qquad=\beta(u_h^k(1-(u_h^k)^{\delta})((u_h^k)^{\delta}-\gamma),\chi)+ (f^k,\chi).
	\end{align}
	From \eqref{2.sdi}-\eqref{2.fd3}, we have
	\begin{align*}
		&(u_h(t_k)-u_h^k,\chi)-(u_h(t_{k-1})-u_h^{k-1},\chi)+\nu \left(\int_{t_{k-1}}^{t_k}\nabla_h u_h(t) \mathrm{~d}t-\Delta t\nabla_h u_h^k,\nabla_h \chi\right) \\&+\eta\left(\int_{t_{k-1}}^{t_k}(K*\nabla_h u_h)(t)\mathrm{~d} t,\nabla_h \chi\right) - \Delta t \eta \left(\sum_{j=1}^{k}\Delta t \omega_{kj}\nabla_h u_h^{j} ,\nabla_h\chi\right) \\&+\alpha \left(\left(\int_{t_{k-1}}^{t_k} B_{CR}(u_h(t)) \mathrm{~d} t\right) - \Delta t B_{CR}(u_h^k),\chi
		\right)=\beta\left(\int_{t_{k-1}}^{t_k}c(u_h(t))\mathrm{~d} t - c(u_h^k),\chi\right).
	\end{align*}
	Take $\chi = u_h(t_k)-u_h^k$ and rearrange the above equation; we achieve
	\begin{align*}
		&\|u_h(t_k)-u_h^k\|_{\L^2}^2+\nu \Delta t \|\nabla_h (u_h(t_k) -u_h^k)\|_{\L^2}^2 + \Delta t \eta \left(\sum_{j=1}^{k}\omega_{kj}\Delta t\nabla_h (u_h(t_j)-u_h^{j}) ,\nabla_h\chi\right) \\&=	(u_h(t_{k-1})-u_h^{k-1},\chi) - \nu \left(\int_{t_{k-1}}^{t_k}\nabla_h u_h(t) \mathrm{~d}t-\Delta t\nabla_h u_h(t_k),\nabla_h \chi\right)\\&\quad + \eta \left(\Delta t \sum_{j=1}^{k}\omega_{kj}\Delta t\nabla_h u_h(t_j) - \int_{t_{k-1}}^{t_k}(K*\nabla_h u_h)(t)\mathrm{~d} t,\nabla_h \chi\right)  \\&\quad-\alpha \left(\int_{t_{k-1}}^{t_k} B_{CR}(u_h(t)) \mathrm{~d} t- \Delta t B_{CR}(u_h(t_k)),\chi
		\right)-\alpha \Delta t\left(B_{CR}(u_h(t_k) ) -  B_{CR}(u_h^k),\chi
		\right)\\&\quad+\beta \left(\int_{t_{k-1}}^{t_k} c(u_h(t)) \mathrm{~d} t - \Delta t c(u_h(t_k)),\chi
		\right)+\beta \Delta t\left(c(u_h(t_k) ) -  c(u_h^k),\chi
		\right).
	\end{align*}
	The first term on right-hand side can be estimated using Cauchy-S\'chwarz and Young's inequality as  
	\begin{align*}
		(u_h(t_{k-1})-u_h^{k-1},u_h(t_k) -u_h(t_k)) \leq \frac{1}{2}\|u_h(t_{k-1})-u_h^{k-1}\|_{\L^2}^2 + \frac{1}{2}\|u_h(t_k) -u_h(t_k)\|_{\L^2}^2\label{2.fde1},
	\end{align*}
	Again using Cauchy-S\'chwarz, we achieve
	\begin{align*}
		&\nonumber\nu \left(\int_{t_{k-1}}^{t_k}\nabla_h u_h(t) \mathrm{~d}t-\Delta t\nabla_h u_h(t_k),\nabla_h (u_h(t_k)-u_h^k)\right)\\
		&\quad\leq \frac{2}{\nu\Delta t}\Big\|\int_{t_{k-1}}^{t_k}\nabla_h u_h(t)\mathrm{~d}t -\Delta t\nabla_h u_h(t_k) \Big\|_{\L^2(\mathcal{T}_h)}^2+\frac{\nu\Delta t}{8}\Big\|\nabla_h (u_h(t_k)-u_h^k)\Big\|_{\L^2(\mathcal{T}_h)}^2,
	\end{align*}
	where the first term on the right hand side can be estimated as 
\begin{align}\label{2.int1}
		\nonumber\frac{1}{\Delta t}\Big\|\int_{t_{k-1}}^{t_k}\nabla_h u_h(t)\mathrm{~d}t-\Delta t\nabla_h u_h(t_k) \Big\|_{\L^2(\mathcal{T}_h)}^2 
		&= \frac{1}{\Delta t}\Big\|\int_{t_{k-1}}^{t_k}\int_{t_k}^{t}\nabla_h \partial_tu_h(s) \mathrm{~d}s\mathrm{~d}t\Big\|_{\L^2(\mathcal{T}_h)}^2
		\\&\leq (\Delta t)^2\int_{t_{k-1}}^{t_k} \| \partial_t \nabla_h u_h(s)\|_{\L^2(\mathcal{T}_h)}^2 \mathrm{~d}s.
	\end{align}
	Estimating the memory term as 
	\begin{align*}
		& \eta \left(\Delta t \sum_{j=1}^{k}\Delta t\omega_{kj}\nabla_h u_h(t_j) - \int_{t_{k-1}}^{t_k}(K*\nabla_h u_h)(t)\mathrm{~d} t,\nabla_h (u_h(t_k)-u_h^k)\right) \\
		& \leq \frac{2\eta^2}{\nu\Delta t} \Big\|\Delta t \sum_{j=1}^{k}\Delta t\omega_{kj}\nabla_h u_h(t_j) - \int_{t_{k-1}}^{t_k}(K*\nabla_h u_h)(t)\mathrm{~d} t\Big\|^2_{\L^2(\mathcal{T}_h)}+\frac{\nu\Delta t}{8}\|\nabla_h (u_h(t_k)-u_h^k)\|^2_{\L^2(\mathcal{T}_h)},
	\end{align*}
	where 
	$\frac{2\eta^2}{\nu\Delta t} \Big\|(\Delta t)^2 \sum_{j=1}^{k}\omega_{kj}\nabla_h u_h(t_j) - \int_{t_{k-1}}^{t_k}(K*\nabla_h u_h)(t)\mathrm{~d} t\Big\|^2_{\L^2(\mathcal{T}_h)}$ can be estimated using \eqref{2.int1} as
	\begin{align*}
		\frac{1}{\Delta t}\Big\|(\Delta t)^2 \sum_{j=1}^{k}\omega_{kj}\nabla_h u_h(t_j) - \int_{t_{k-1}}^{t_k}(K*\nabla_h u_h)(t)\mathrm{~d} t\Big\|_{\L^2(\mathcal{T}_h)}^2 
	\\	\leq T(\Delta t)^3\sum_{j=1}^k \overline{K}_{kj}^2\int_{t_{j-1}}^{t_j}\|\partial_t\nabla_h u_h(\tau)\|_{\L^2(\mathcal{T}_h)}^2\mathrm{~d}\tau,
	\end{align*}
	where $\overline{K}_{kj} =  \frac{1}{\Delta t}\int_{t_{k-1}}^{t_k}\int_{t_{j-1}}^{\min(t,t_j)}|K(t-s)| \mathrm{~d}s\mathrm{~d} t$.
	Applying integration by parts, inverse inequality and Cauchy Schwarz  inequality gives 
	\begin{align}\label{2.nl1}
		\nonumber	&\alpha \left(\int_{t_{k-1}}^{t_k} B_{CR}(u_h(t)) \mathrm{~d} t - \Delta t B_{CR}(u_h(t_k)),u_h(t_k)-u_h^k
		\right) \\
		&\leq \frac{2\alpha^2}{\nu(\delta+1)^2\Delta t}\Big\|\int_{t_{k-1}}^{t_k} u_h^{\delta+1}(t) \mathrm{d}t - \Delta t u_h^{\delta+1}(t_k)\Big\|_{\L^2}^2+\frac{\nu\Delta t}{8}\Big\|\nabla_h(u_h(t_k)-u_h^k)\Big\|^{2}_{\L^2(\mathcal{T}_h)}.
	\end{align}
	Now, estimating  $\frac{1}{\Delta t}\Big\|\int_{t_{k-1}}^{t_k}u_h^{\delta+1}(t)\mathrm{~d}t-\Delta t u_h^{\delta+1}(t_k) \Big\|_{\L^2}^2 $ similar to \eqref{2.int1} as   
	\begin{align*}
		&\frac{1}{\Delta t}\Big\|\int_{t_{k-1}}^{t_k}u_h^{\delta+1}(t)\mathrm{~d}t-\Delta t u_h^{\delta+1}(t_k)\Big\|_{\L^2}^2   
		\leq(\Delta t)^2(\delta+1)^2\int_{t_{k-1}}^{t_k} \|u_h^{\delta}(s)\partial_tu_h(s)\|_{\L^2}^2\mathrm{~d}s.
	\end{align*}
	The non-linear reaction term can be estimated as 
	\begin{align}\label{2.33}
		\nonumber\left(\int_{t_{k-1}}^{t_k} c(u_h(t)) \mathrm{~d} t\right)& - \Delta t c(u_h(t_k))= \beta (1+\gamma)\left(\int_{t_{k-1}}^{t_k} u_h^{\delta+1}(t) \mathrm{~d} t-\Delta t u_h^{\delta+1}(t_k)\right)\\&-\beta\gamma\left(\left(\int_{t_{k-1}}^{t_k} u_h(t) \mathrm{~d} t\right)-\Delta t u_h(t_k)\right)-\beta\left(\int_{t_{k-1}}^{t_k} u_h^{2\delta+1}(t) \mathrm{~d} t-\Delta t u_h^{2\delta+1}(t_k)\right).
	\end{align}
To estimate the second term, we use Cauchy-S\'chwarz inequality and the approach similar to \eqref{2.int1} as
	\begin{align}
		\nonumber\left(\int_{t_{k-1}}^{t_k} u_h(t) \mathrm{~d} t-\Delta t u_h(t_k),u_h(t_k)-u_h^k \right) 
		&\leq \frac{(\Delta t)^2}{\nu}\int_{t_{k-1}}^{t_k} \| \partial_tu_h(s)\|_{\L^2}^2 \mathrm{~d}s +\frac{\nu\Delta t}{8}\|(u_h(t_k)-u_h^k)\|^{2}_{\L^2}.
	\end{align}
    The same approach discussed in \eqref{2.nl1} gives
	\begin{align*}
		\nonumber&\beta(1+\gamma)\left(\int_{t_{k-1}}^{t_k} u_h^{\delta+1}(t) \mathrm{~d} t-\Delta t u_h^{\delta+1}(t_k), u_h(t_k)-u_h^k  \right)\\
		&\quad\leq  \frac{(\Delta t)^2(1+\gamma)^2\beta^2(\delta+1)^2}{\nu}\int_{t_{k-1}}^{t_k} \|u_h^{\delta}(s)\partial_tu_h(s)\|_{\L^2}^2 + \frac{\nu\Delta t}{8}\|u_h(t_k)-u_h^k\|^{2}_{\L^2}.
	\end{align*}
	The final term of \eqref{2.33} can be estimated as
	\begin{align*}
	&\Big|\beta\left(\int_{t_{k-1}}^{t_k} u_h^{2\delta+1}(t) \mathrm{~d} t-\Delta t u_h^{2\delta+1}(t_k), u_h(t_k)-u_h^k  \right)\Big|\\
	&\leq\frac{2\beta^2(2(\delta+1))^2(\Delta t)^2}{\nu}\|u_h\|_{\L^{\infty}(0,T; \L^{2(\delta+1)}(\Omega))}^{2\delta}\int_{t_{k-1}}^{t_k}\|u_h^{\delta}(t)\partial_tu_h(t)\|^2_{\L^2}  \mathrm{~d} t +\frac{\nu\Delta t}{8} \|\nabla_h(u_h(t_k)-u_h^k)\|^2_{\L^{2}}.
\end{align*}
First, we use Taylor's formula, Inverse and H\"older's inequalities in  $-\alpha\Delta t(B_{CR}(u_h(t_k))-B_{CR}(u_h^k),w)$. Then, applying discrete Gagliardo-Nirenberg \cite{CHF}, interpolation  and Young's inequalities yields
\begin{align*}
	&-\alpha \Delta t\left(B_{CR}(u_h(t_k) ) -  B_{ CR}(u_h^k),u_h(t_k)-u_h^k
	\right)\\&=\frac{\alpha}{\delta+1} \left((u_h(t_k)^{\delta+1}-(u_h^k)^{\delta+1})\left(\begin{array}{c}1\\\vdots\\1\end{array}\right),\nabla_h w\right)\nonumber\\
	&\leq \frac{\nu}{8}\|\nabla_h w\|_{\L^2(\mathcal{T}_h)}^2+C(\alpha,\nu)\left(\|u_h(t_k)\|_{\L^{2(\delta+1)}}^{\frac{4\delta(\delta+1)}{(2-d)\delta+2}}+\|u_h^k\|_{\L^{2(\delta+1)}}^{\frac{4\delta(\delta+1)}{(2-d)\delta+2}}\right)\|w\|_{\L^2}^2.
\end{align*}
where $C(\alpha, \nu) = \left(\frac{2((2+d)\delta+2)}{\nu(\delta+1)}\right)^{\frac{(2-d)\delta+2}{(2+d)\delta+2}}\times\left(\frac{(2-d)\delta+2}{4(\delta+1)}\right)(2^{\delta-1}\alpha)^{\frac{4(\delta+1)}{(2-d)\delta+2}}$.

Combining the above estimates and using the calculations similar to  \eqref{2.30}, then summing overall $k$,
 $k = 1,2,\cdots, N$ , and using the positivity of the kernel \eqref{2.pk}, we obtain
 	\begin{align*}
	&\|u_h(t_N)-u_h^N\|_{\L^2}^2+\nu|\!|\!|u_h(t_k) -u_h^k|\!|\!|_{CR}^2\nonumber\\& \leq\frac{4\alpha^2(\Delta t)^2}{\nu}\int_{0}^{T} \| \partial_t \nabla_h u_h(s)\|_{\L^2}^2 \mathrm{~d}s
	+\frac{2(\Delta t)^2\beta^2\gamma^2}{\nu}
	\int_{0}^{T} \|\partial_tu_h(s)\|_{\L^2}^2 \mathrm{~d}s   \nonumber \\&\quad+2\Delta t \sum_{k=1}^N\bigg(C(\alpha,\nu)\left(\|u_h(t_k)\|_{\L^{2(\delta+1)}}^{\frac{4\delta(\delta+1)}{(2-d)\delta+2}}+\|u_h^k\|_{\L^{2(\delta+1)}}^{\frac{4\delta(\delta+1)}{(2-d)\delta+2}}\right)+\frac{\beta}{2}2^{2\delta}(1+\gamma)^2(\delta+1)^2\nonumber\\&\quad + \frac{3\nu}{8}\bigg)\|u_h(t_k)-u_h^k\|_{\L^2}^2+\frac{2(\Delta t)^2}{\nu} \left(2\alpha^2 + C \beta^2(1+\gamma)^2(\delta+1)^2\right)\int_{0}^{T} \|u_h(s)^{\delta}\partial_tu_h(s)\|_{\L^2}^2\mathrm{~d}s\nonumber \\&\quad + \frac{4\eta^2T(\Delta t)^3}{\nu}\sum_{k=1}^N\sum_{j=1}^k \overline{K}_{kj}^2\int_{t_{j-1}}^{t_j}\|\partial_t\nabla_h u_h(\tau)\|_{\L^2(\mathcal{T}_h}^2\mathrm{~d}\tau.\label{329}
\end{align*}
Using 
 	\begin{align*}
		T(\Delta t)^3\sum_{k=1}^N\sum_{j=1}^k \omega_{kj}^2\int_{t_{j-1}}^{t_j}\|\partial_t\nabla_h u_h(\tau)\|_{\L^2(\mathcal{T}_h)}^2\mathrm{~d}\tau 
		& \leq \sup_{k,j} \omega_{kj}^2 T(\Delta t)^2 \int_{0}^{T}\|\partial_t\nabla_h u_h(\tau)\|_{\L^2(\mathcal{T}_h)}^2\mathrm{~d}\tau,
	\end{align*}
and Gronwall's inequality implies that 
	\begin{align}\label{2.fdest}
		&\nonumber\|u_h(t_N)-u_h^N\|_{\L^2}^2+ \nu|\!|\!|u_h(t_k) -u_h^k|\!|\!|_{CR}^2\\&\nonumber \leq C(\Delta t)^2\left(\int_{0}^{T} \| \partial_t \nabla_h u_h(s)\|_{\L^2(\mathcal{T}_h)}^2 \mathrm{~d}s +\int_{0}^{T} \|u_h(s)^{\delta}\partial_tu_h(s)\|_{\L^2}^2+\frac{4\eta^2}{\nu}\sup_{k,j} \omega_{kj}^2 \int_{0}^{T} \|\partial_tu_h(s)\|_{\L^2}^2 \mathrm{d}s \right)\\&\quad  \times  \exp\left\{2\Delta t\left[C(\alpha,\nu)\left(\|u_h(t_k)\|_{\L^{2(\delta+1)}}^{\frac{4\delta(\delta+1)}{(2-d)\delta+2}}+\|u_h^k\|_{\L^{2(\delta+1)}}^{\frac{4\delta(\delta+1)}{(2-d)\delta+2}}\right)+\frac{\beta}{2}2^{2\delta}(1+\gamma)^2(\delta+1)^2 + \frac{3\nu}{8}\right]\right\}.
	\end{align}
	\vskip 2mm
	\noindent\textbf{Step 2:} \emph{Estimate for any $t\in [t_{k-1},t_k]$.} First, we define the following linear interpolation for the semi-discrete solution $u_h$, \cite[Section 3.1]{YGU}):
	$$\mathcal{I} u_h(t_k) =  u_h(t_{k-1})+\left(\frac{t-t_{k-1}}{\Delta t}\right)(u_h(t_{k})-u_h(t_{k-1})), \qquad \qquad \forall t\in [t_{k-1},t_k].$$
	Then, the error term $u_h-u_{kh}$ is divided as, $u_h-u_{kh} = u_h -\mathcal{I}u_h + \mathcal{I}u_h -u_{kh}$. A simple application of a triangle inequality gives
	\begin{align*}
		\|u_h - u_{kh}\|_{\L^2(0,T;\H^1(\mathcal{T}_h))}^2 \leq   2\|u_h-\mathcal{I}u_h\|_{\L^2(0,T;\H^1(\mathcal{T}_h))}^2 +2 \|\mathcal{I}u_h -u_{kh}\|_{\L^2(0,T;\H^1(\mathcal{T}_h))}^2.
	\end{align*}
Invoking \cite[Lemma 3.2]{YGU} for the first part, we attain
	\begin{align}\label{2.sim}
		\nonumber&\|u_h-\mathcal{I} u_h\|_{\L^2(0,T;\H^1(\mathcal{T}_h))}^2 = \sum_{i=1}^N\int_{t_{i-1}}^{t_i}\|\nabla_h(u_h-\mathcal{I} u_h)\|_{\H^1(\mathcal{T}_h)}^2 \leq C(\Delta t )^2 \int_{0}^{T}\left\|\partial_tu_h\right\|_{\H^1(\mathcal{T}_h)}^2 \mathrm{d}t,\\&
		\|\mathcal{I}u_h -u_{kh}\|_{\L^2(0,T;\H^1(\mathcal{T}_h))}^2\leq\ \sum_{i=1}^N\int_{t_{i-1}}^{t_i} \|\mathcal{I}u_h(t) -u_{kh}(t)\|_{\H^1(\mathcal{T}_h)}\leq C  \sum_{i=1}^N\Delta t \|u_h(t_i) -u_h^i\|_{\H^1(\mathcal{T}_h)}.
	\end{align}
	Using the triangle inequality gives
	\begin{align*}
		\|u - u_{kh}\|_{\L^{\infty}(0,T;\L^2(\Omega))}^2 \leq 2 \|u_h-\mathcal{I}u_h\|_{\L^{\infty}(0,T;\L^2(\Omega))}^2 + 2\|\mathcal{I}u_h -u_{kh}\|_{\L^{\infty}(0,T;\L^2(\Omega))}^2.
	\end{align*}
	As in \eqref{2.sim}, by using \cite[Corollary 3.1]{YGU}, we achieve
	\begin{align}\label{2.sim1}
		\nonumber&\|u_h-\mathcal{I} u_h\|_{\L^{\infty}(0,T;\L^2(\Omega))}^2 \leq\sup_{1\leq i \leq N}\Big(\sup_{t_{i-1}\leq t\leq t_{i}}\|u_h-\mathcal{I} u_h\|^2_{\L^2(\Omega)}\Big) \leq C(\Delta t )^2 \|u_h\|_{W^{1,\infty}(0,T ; \L^2(\Omega))}^2,
		\\& \|\mathcal{I}u_h -u_{kh}\|_{\L^{\infty}(0,T;\L^2(\Omega))}^2 \leq\sup_{1\leq i \leq N}\Big(\sup_{t_{i-1}\leq t\leq t_{i}}\|\mathcal{I}u_h -u_{kh}\|^2_{\L^2(\Omega)}\Big) \leq C\sup_{1\leq i \leq N}\Big(\|u_h(t_i) -u_h^i\|^2_{\L^2(\Omega)}\Big).
	\end{align}
	Combining \eqref{2.sim}-\eqref{2.sim1} with \eqref{2.fdest} leads to the desired result.
\end{proof}
Finally, we state the main theorem of this section 
\begin{theorem}\label{2.thm7.3}
	 For the initial data $u_0\in \H^2(\Omega)\cap \H_0^1(\Omega)$ and $f\in\H^1(0,T;\L^2(\Omega)),$ we have  as $\Delta t, h\rightarrow 0$ the finite element approximation $u_{kh}$ converges to $u$ . In addition, the following estimate is satisfied:
\begin{align}\label{2.fdes}
		\nonumber \|u- u_{kh}\|_{\L^{\infty}(0,T;\L^2(\Omega))}^2 + |\!|\!|u - u_{kh}|\!|\!|_{CR}^2&\leq C(f,u_0)(\eta^2(\Delta t)^2\sup_{k,j} \omega_{kj}^2 +(\Delta t)^2 + h^2).
	\end{align}
\end{theorem}
\begin{proof}
	The proof follows directly from Theorem \ref{2.SDNCFEM} and Lemma \ref{2.thm7.2}.
\end{proof}
\subsection{Discontinuous Galerkin method}\label{2.secDG}
Additional to the mesh notation used so far, we define some notations for DG formulation.
The shared edge between the two mesh cells $K_\pm$ is denoted by, $E=K_+\cap K_-\in\mathcal{E}^i_h$. Moreover, the traces of functions $w\in C^0(\mathcal{T}_h)$, on $E$ of $K_\pm$ are denoted by $w_{\pm}$
respectively. The average operator $\{\!\{\cdot\}\!\}$ and the jump operator on edge $E$ are defined as:
\begin{align*}
	\{\!\{w\}\!\}=\frac{1}{2}(w_+ + w_-) \quad  \text{and}	\quad[\![w]\!]={w_+\mathbf{{n}_+} + w_-\mathbf{{n}_-}},
\end{align*}
respectively. If ${w}\in C^1(\mathcal{T}_h),$ we define $	[\![\partial{w}/\partial\mathbf{{n}}]\!]=\nabla({w}_+ -{w}_-)\cdot\mathbf{{n}_+}, $
where $\mathbf{{n}_\pm}$ represents the unit outward normal vectors for the respective mesh cells $K_\pm$. If $E\in K_+\cap\partial\Omega$, then we have ${[\![{w}]\!]={w}_+\mathbf{{n}_+}}$ and $\{\!\{w\}\!\}=w_+$. We denote the exterior trace of the function $u$ by $u^e$. For the boundary edges, we choose  $u^e =0.$ The local gradient on each $K\in\mathcal{T}_h$ is denoted by the notation $\nabla_h$, with $(\nabla_h{w})|_K = \nabla({w}|_K)$ . The discrete space for DG formulation is defined as 
\begin{align}\label{2.dgsubspace1}
	{V}_h^{DG}=\{{v}\in \L^2(\Omega): \forall\  K\in \mathcal{T}_h : {v}|_K \in \mathcal{P}_1(K)\},
\end{align}
where $\mathcal{P}_1(K)$ denotes the space of polynomials of degree $1$ on $K.$
\subsubsection{Semi-discrete DGFEM}
In this context, the semi-discrete weak formulation of \eqref{2.GBHE} is given by: Find $u^{DG}_h\in V_h^{DG}$, for $t\in(0,T)$ such that
	\begin{align}\label{2.dgweakform}
	\nonumber(\partial_tu^{DG}_h(t), \chi(t))+A_{DG}(u^{DG}_h(t),\chi(t))+\eta(( K*\nabla_h u^{DG}_h)(t),\nabla_h \chi(t))&=( f(t),\chi(t)),\\
	(u^{DG}_h(0),\chi(t))&=(u_h^0,\chi(t)), 
	\end{align}
$\forall$ $\chi \in V_h^{DG}$, where
\begin{align}\label{2.ADGf} 
A_{DG}(u,v) = \nu a_{DG}(u,v) + \alpha b_{DG}(u,u,v)-\beta(c(u),v),
\end{align}
with
\begin{align}\label{2.adg}
\nonumber a_{DG}({u},{v})&=(\nabla_h {u}, \nabla_h {v})-\sum_{E\in\mathcal{E}_h}\int_E {\{\!\{\nabla_h {u}\}\!\}}\!\cdot\![\![{v}]\!]\mathrm{~d}{s}\\&\quad-\sum_{E\in\mathcal{E}_h}\int_E {\{\!\{\nabla_h {v}\}\!\}}\!\cdot\![\![{u}]\!]\mathrm{~d}{s}
	+\sum_{E\in\mathcal{E}_h}\int_E\gamma_h [\![{u}]\!]\!\cdot\![\![{v}]\!]\mathrm{~d}{s},
\end{align}
and 
	\begin{align}\label{2.bDG}
	\nonumber	b_{DG}(\mathbf{w};u,v)=\frac{1}{\delta+2}\Bigg(& \sum_{K\in \mathcal{T}_h} \int_{K} \mathbf{w}\cdot \nabla u v \mathrm{~d}x 
	+\sum_{K\in \mathcal{T}_h} \int_{\partial K} \hat{\mathbf{w}}_{h,u}^{up}{v} \mathrm{~d}s\\& -\sum_{K\in \mathcal{T}_h} \int_{K} \mathbf{w}\cdot \nabla v {u} \mathrm{~d}x 
	-\sum_{K\in \mathcal{T}_h} \int_{\partial K} \hat{\mathbf{w}}_{h,v}^{up}{u} \mathrm{~d}s\Bigg).
\end{align}
Here $\gamma_h=\frac{\gamma}{h_E}$ and the upwind flux
\begin{align*}
	\hat{\mathbf{w}}_{h,u}^{up}=\frac{1}{2}\left[\mathbf{w}\cdot\mathbf{n}_K -|\mathbf{w}\cdot \mathbf{n}_K|\right](u^e\!-\!u).
\end{align*}
with $\mathbf{w}=(w,w)^T$. 
The length of the edge $E$ is represented by the parameter $h_E$. In order to guarantee the stability of the formulation, the penalty parameter $\gamma$ is selected to be sufficiently large (see, e.g.,  \cite{Arn}). The following discrete norm is used for further error analysis: 
\[
|\!|\!|{ v}|\!|\!|_{DG}^2:= \sum_{K\in\mathcal{T}_h}\|\nabla_h v\|_{\L^2(\mathcal{T}_h)}^2 + \sum_{E\in\mathcal{E}_h}\gamma_h\|[\![v]\!]\|_{\L^2(E)}^2.\]

\begin{lemma}\label{2.dglem11}[Coercivity and Stability]
	\begin{enumerate}
		\item 	For any $ v \in V_{h}^{DG}$, the operator $a_{DG}$ is coercive, i.e.,
		\[ a_{DG}(v,v) \ge \alpha_a|\!|\!|{ v}|\!|\!|^2_{DG}\label{2.DGC},\] 
		for a positive constant, $\alpha_a\geq 0.$
		\item 	Assume that $f\in\L^2(0,T;\L^2(\Omega))$ and $u_0\in \L^2(\Omega)$, then the semi-discretized solution $u_h$ of the \eqref{2.GBHE} defined in \eqref{2.ncweakformfd} is stable. In other words, we have
		\begin{align}\label{2.DG13} 
			&\sup_{0\leq t\leq T}\|u_h(t)\|_{\L^2}^2+\nu\int_0^T\|\nabla_h u_h(t)\|_{\L^2(\mathcal{T}_h)}^2\mathrm{~d} t\leq\left(\|u_0\|_{\L^2}^2+\frac{1}{\nu}\int_0^T\|f(t)\|_{\L^2}^2\mathrm{~d}t\right)e^{\beta(1+\gamma^2)T}.
		\end{align}
	\end{enumerate}
\end{lemma}
\begin{proof}
	The proof of the coercivity of $a_{DG}(\cdot,\cdot)$ directly follows from  \cite[Section 3]{Arn}. Note that $b_{DG}(u_h,u_h,u_h) = 0$. The rest of the proof of the stability is identical to Lemma \ref{2.stabilitysd}.
\end{proof}
In the next lemma, we discuss the result required for the error estimates,
\begin{lemma}\label{2.lem111}
	There holds:
	\begin{align*}
		-\alpha[b_{DG}(u_h;u_h,w)-b_{DG}(v_h;v_h,w)]&\le  \frac{\nu}{2}|\!|\!| w|\!|\!|^2_{DG}+C(\alpha,\nu)\left(\|u_h\|^{\frac{8\delta}{4-d}}_{\L^{4\delta}}+\|v_h\|^{\frac{8\delta}{4-d}}_{\L^{4\delta}}\right)\|w\|_{\L^2}^2,\\
		A_{DG}(u_h,w)-A_{DG}(v_h,w)&\ge  \frac{\nu}{2}|\!|\!| w|\!|\!|^2_{DG} +\frac{\beta}{4}(\|{u}^{\delta}_hw\|_{\L^2}^2+\|v^{\delta}_hw\|_{\L^2}^2)\nonumber\\
		&\quad+\left(\beta\gamma-C(\beta,\alpha,\delta) - C(\alpha,\nu)\Big(\|u_h\|^{\frac{8\delta}{4-d}}_{\L^{4\delta}}+\|v_h\|^{\frac{8\delta}{4-d}}_{\L^{4\delta}}\Big)\right)\|w\|_{\L^2}^2,
	\end{align*}
	where $u_h,v_h\in V_{h}^{DG}$, $w=u_h-v_h$, $C(\alpha, \nu) = \left(\frac{4+d}{4\nu}\right)^{\frac{4+d}{4-d}}\left(\frac{4-d}{8}\right)(\frac{2^{\delta-1}C\alpha}{(\delta+2)(\delta+1)})^{\frac{4-d}{8}}$ and  $C(\beta,\alpha,\delta)= \frac{\beta}{2}2^{2\delta}(1+\gamma)^2(\delta+1)^2$ is a positive constant depending on parameters.
\end{lemma}
\begin{proof}
	The idea of proof is similar to Lemma \ref{2.crnclem111}.
\end{proof}
Finally, we 
state the a priori error estimate for the semi-discrete DG approximation.
\begin{theorem}\label{2.FDDGFEM}
	Assume that $u$ be the solution of \eqref{2.weaksolution}, then the error incurred by the DGFEM approximation $u_h^{DG}$ tends to $0$ as  $h\rightarrow 0$, i.e.,
	\begin{align*}
	\|u_h^{DG}-u\|_{\L^{\infty}(0,T;\L^2(\Omega))}^2 +	|\!|\!|u_h^{DG}-u|\!|\!|_{DG}^2\leq C\bigg\{\|u^h_0-u_0\|_{\L^2}^2+h^2~\Theta(u)\bigg \},
\end{align*}
	where $C$ is a positive constant, independent of $h$, and $\Theta(u)$ is given in \eqref{2.th}.
\end{theorem}
\begin{proof}
Using the formulation \eqref{2.dgweakform}, we have 
$$
 ({\partial_t}u_h^{DG}(t), \chi)+A_{DG}(u^{DG}_h(t), \chi)+\eta\Delta t\int_0^tK(t-s)a_{DG}( u_h^{DG}(s), \chi) \mathrm{~d}s-(f(t),\chi) = 0,  
$$
	$\forall ~\chi\in V_h^{DG}$. If $u\in \H_0^1(\Omega)\cap \H^2(\Omega)$ satisfies \eqref{2.ncweakform}, then we have that, $\forall ~t\in(0,T),$
	$$
 ({\partial_t}u(t), \chi)+A_{DG}(u(t), \chi)+\eta\Delta t \int_0^tK(t-s)a_{DG}( u(s), \chi) \mathrm{~d}s-(f(t),\chi) = 0.
$$
Subtracting the above two equations, we get 
\begin{align}
\nonumber({\partial_t}(u_h^{DG}(t)-u(t)), \chi)+A_{DG}(u_h^{DG}(t), \chi) - A_{DG}( u(t), \chi) \\+\eta\Delta t\int_0^tK(t-s)a_{DG}( u_h^{DG}(s)-u(s), \chi)\mathrm{~d}s = 0.	
\end{align}
Making a specific choice $\chi = u_h^{DG} -W $, and rewriting $u_h^{DG}-u = u_h^{DG}-W+W-u$ gives
\begin{align}
	\nonumber&({\partial_t}(u_h^{DG}(t)-W(t)),\chi)+A_{DG}(u_h^{DG}(t),\chi) - A_{DG}(W(t),\chi) \\& +\eta\Delta t\int_0^tK(t-s)a_{DG}( u_h^{DG}(s)-W(s), \chi) \mathrm{~d}s = - ({\partial_t}(W(t)-u(t)),\chi) \nonumber\\&-A_{DG}(W(t),\chi)+A_{DG}(u(t), \chi)-\eta\Delta t\int_0^tK(t-s)a_{DG}( W(s)-u(s),\chi)\mathrm{~d}s.
\end{align}
To prove the desired result, we proceed similarly as in Theorem \ref{2.SDNCFEM} and an application of Lemma \ref{2.lem111}.
\end{proof}
\subsubsection{Fully-discrete DGFEM}
The fully-discrete weak formulation of (\ref{2.GBHE}), is given as: Find $(u^{DG}_h)^k= u_h^k\in V_h^{DG}$ (for simplicity of notation take $(u^{DG}_h)^k= u_h^k$), such that
\begin{align}\label{2.dgfdweakform}
(\overline{\partial}u_h^k, \chi)+ A_{DG}(u_h^k(t),\chi)+\eta \left(\sum_{j=1}^{k}\omega_{k j}a_{DG}\left(u_h^j (t),v\right)\right) = (f^k,\chi),\nonumber&\\
(u_h(0),\chi)=(u_h^0,\chi),
\end{align} 
where, $\omega_{k j}=\frac{1}{(\Delta t)^2 } \int_{t_{k-1}}^{t_k} \int_{t_{j-1}}^{\min \left(t, t_j\right)} K(t-s) \mathrm{~d}s \mathrm{~d}t$, for $1\leq k\leq N$, $f^k = (\Delta t)^{-1}\int_{t_{k-1}}^{t_k} f(s)  \mathrm{~d}s$ and $A_{DG}$ is  defined in \eqref{2.ADGf}.
Similar to \eqref{2.gfds}, we define DG approximated solution as
\begin{align}\label{2.gfdsdg}
	u^{DG}_{kh} |_{[t_{k-1},t_k]} = u_h^{k-1}+\left(\frac{t-t_{k-1}}{\Delta t}\right)(u_h^{k}-u_h^{k-1}), \qquad 1\leq k\leq N, \qquad \qquad \forall~ t\in [t_{k-1},t_k]. 
\end{align}
The error estimates for the formulation \eqref{2.dgfdweakform} are discussed in the next two results.
\begin{lemma}\label{2.SDGthm7.2} 
	Let $u$ satisfy the hypothesis of Lemma \ref{2.thm7.2}. Then, the
following bound holds:
	\begin{align*}
	\nonumber\|u_h^{DG} - u^{DG}_{kh}\|_{\L^{\infty}(0,T;\L^2(\Omega))}^2 +|\!|\!|u_h^{DG} - u^{DG}_{kh}|\!|\!|_{DG}^2&\leq C(f,u_0)\Big((\Delta t)^2 + \eta^2(\Delta t)^2\sup_{k,j} \omega_{kj}^2\Big).
\end{align*}
\end{lemma}
\begin{proof}
	The idea of the proof follows, similar to the Lemma \ref{2.thm7.2}.
\end{proof}
\begin{theorem}\label{2.DGthm7.3}
	Let $u$ satisfy the hypothesis of Lemma \ref{2.thm7.2}. Then, the
 following bound holds:
	\begin{align}\label{2.fdesdg}
		\nonumber \|u- u^{DG}_{kh}\|_{\L^{\infty}(0,T;\L^2(\Omega))}^2 + |\!|\!|u - u^{DG}_{kh}|\!|\!|_{DG}^2&\leq C(f,u_0)\Big((\Delta t)^2 + \eta^2(\Delta t)^2\sup_{k,j} \omega_{kj}^2\Big).
	\end{align}
where $C$ is a constant independent of $\Delta t$ and $h$.
\end{theorem}
\begin{proof}
Combining Lemma \ref{2.SDGthm7.2} and Theorem \ref{2.FDDGFEM} leads to the stated result.
\end{proof}

\section{Numerical studies}\label{2.secnum}\setcounter{equation}{0} 
In this section, we present numeric findings to substantiate the results established in Theorem \ref{2.thm7.3} and Theorem \ref{2.DGthm7.3}. These computations were performed using the open-source finite element library FEniCS \cite{ABJ}.

In all the examples, we discretize the time derivative using a backward Euler discretization scheme and space using CRFEM or DGFEM. We adopt a temporal discretization scheme with uniform time stepping, $\Delta t = \frac{T}{M},$  such that $t_k = k\Delta t$, where $T$ is total time and $M$ is the number of time steps. The spatial discretization parameter is denoted as $h$. In all the experiments we set $\Delta t \propto h$.
\subsection{Weakly singular kernel}\label{3.1}
Consider the problem \eqref{2.GBHE} defined on the domain $\Omega \times (0,T) = (0,1)^d\times (0,1)$. For the particular choice of the Kernel $K(t) =\frac{1}{\sqrt{t}},$  the approximated solutions $u_h^{CR}$ and $u_h^{DG}$ are obtained using the positive quadrature rule for the kernel term. The error incurred between the numerical solution and the exact solution in 2D and 3D are validated for two different expressions of the exact solution.
\begin{align*}
	\text{Type I:   }u = (t^3-t^2+1)\prod\limits_{i=1}^{d}\sin(\pi x_i),  \qquad 	\text{Type II:   }u = t\sqrt{t}\prod\limits_{i=1}^{d}\sin(2\pi x_i).
\end{align*}
	\begin{figure}
	\begin{center}
		\includegraphics[width=0.450\textwidth]{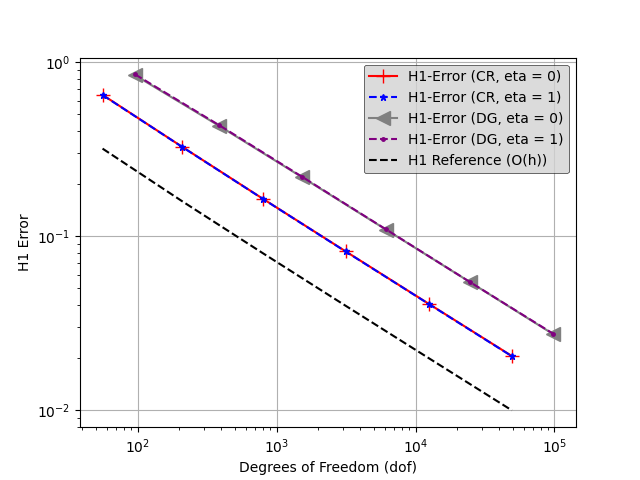}
		\includegraphics[width=0.450\textwidth]{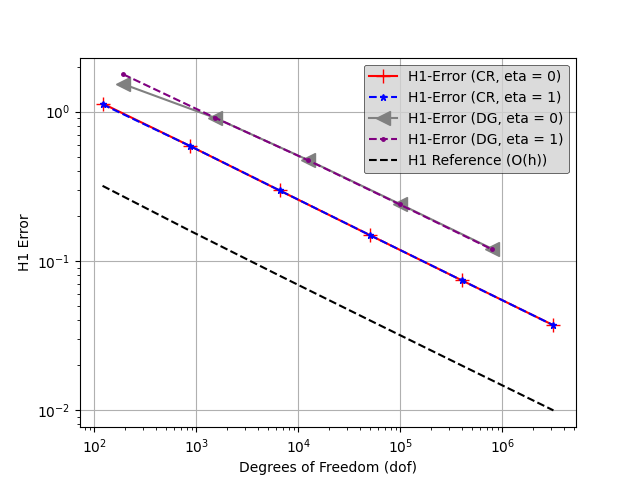}
		\caption{Rate of convergence plot for the numerical solutions $u_h^{CR}$ and  $u_h^{DG}$for the solution defined in Type I   in 2D and 3D.}\label{fig1}
	\end{center}
\end{figure}
	\begin{figure}
	\begin{center}
			\includegraphics[width=0.450\textwidth]{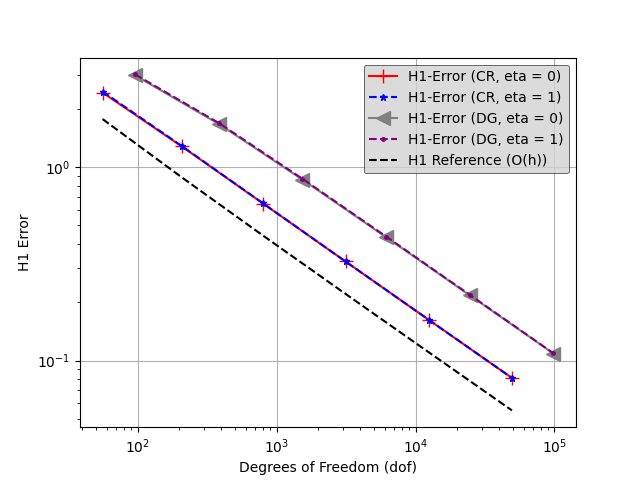}
			\includegraphics[width=0.450\textwidth]{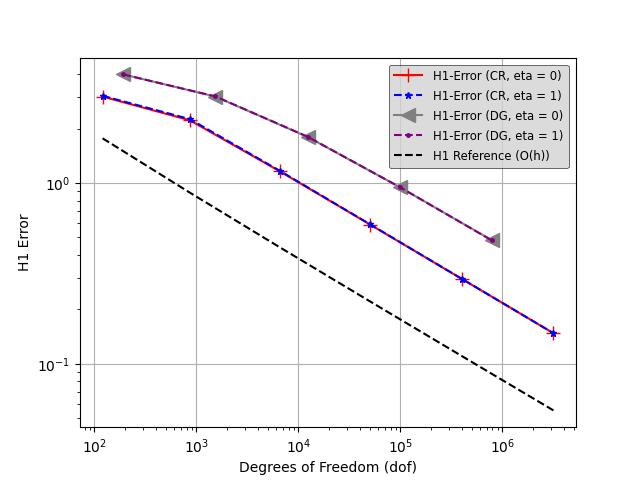}
		\caption{Rate of convergence plot for the numerical solutions $u_h^{CR}$ and  $u_h^{DG}$ for the solution defined in Type II   in 2D and 3D  .}\label{fig2}
	\end{center}
\end{figure}
We set the values of parameters as $\alpha  = \delta = \nu= \beta = 1,$ and $\gamma= 0.5$. Figures \ref{fig1} and \ref{fig2} represent the plot of error in energy norm against degree of freedom for Type I and Type II, respectively. The error in energy norm decreases with the rate of $O(h)$. 
A maximum number of three Newton iterations is demanded to acquire the desired tolerance of $10^{-10}$.
	\subsection{Application to the fractional time derivative}
	The proposed theory in this paper also holds for the following time fractional GBHE with memory given by
\begin{align}\label{2.GBHEFr}
	\mathcal{L} u(x,t)  + \partial_t^{\mu} u(x,t)=f(x,t),  \quad (x,t)\in\Omega\times(0,T),
\end{align}
where, $\Omega= (0,1)^d$ and $T =1$. The expression $\partial_t^\mu$ 
denotes the left-sided Caputo fractional derivative (Pg. 81; \cite{KST} and \cite{BRZ}) of order $\mu\ge 0$ with respect to $t$ defined as:
$$
\partial_t^{\mu} u(t) = \frac{1}{\Gamma(1-\mu)}\int_0^t \frac{1}{(t-\tau)^{\mu}} \frac{\mathrm{du(\tau)}}{\mathrm{d}\tau} \mathrm{d}\tau,
$$
where $\Gamma(.)$ represents the Gamma function. 
 The discretization of the fractional derivative term is carried out in a similar manner to that of the memory term. The plots of the error estimates (Figure \ref{fig3} and Figure \ref{fig4}) demonstrate the first-order convergence for a fractional derivative of order, $\mu = \frac{1}{2}$ and the weakly singular kernel  $K(t)= \frac{1}{\sqrt{t}}$ for the solutions defined in \eqref{3.1}.
	\begin{figure}
	\begin{center}
			\includegraphics[width=0.450\textwidth]{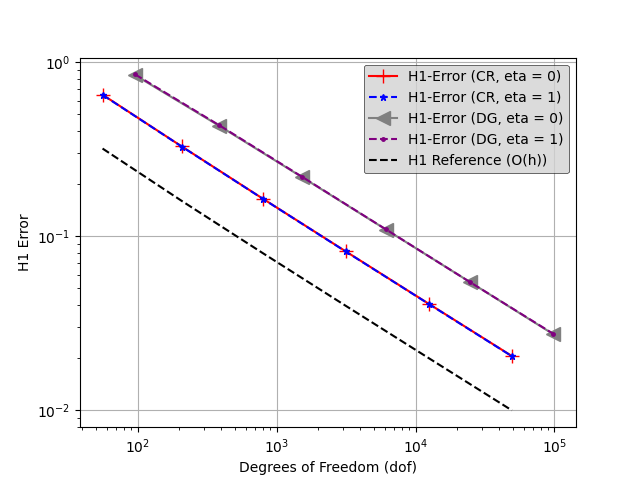}
			\includegraphics[width=0.450\textwidth]{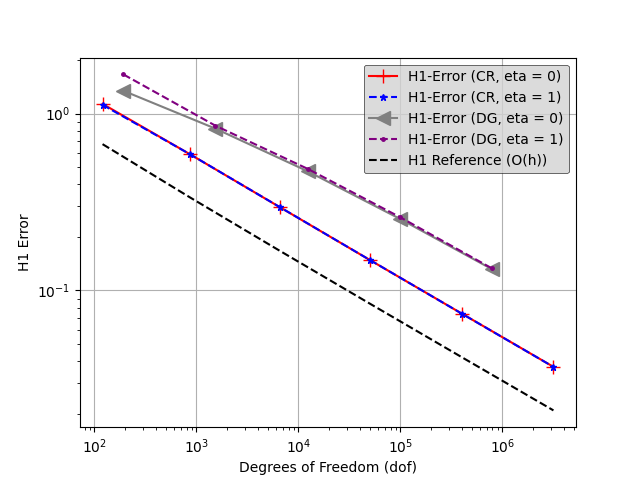}
	\caption{Rate of convergence plot for the numerical solutions $u_h^{CR}$ and  $u_h^{DG}$ for the Caputo fractional time derivative term for the solution defined in Type I   in 2D and 3D .}\label{fig3}
\end{center}
	\end{figure}
	\begin{figure}
	\begin{center}
		\includegraphics[width=0.450\textwidth]{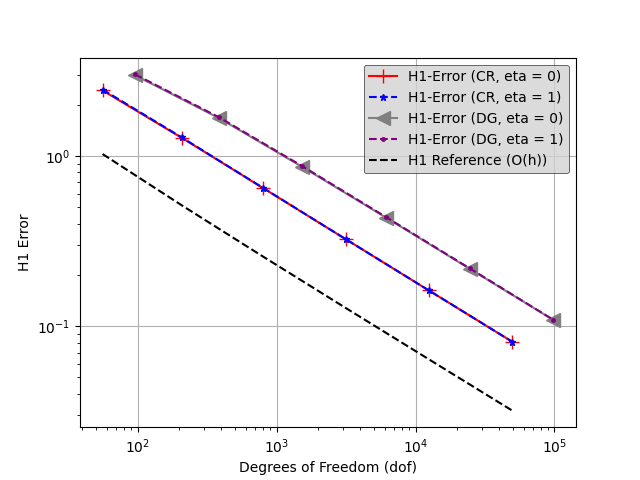}
		\includegraphics[width=0.450\textwidth]{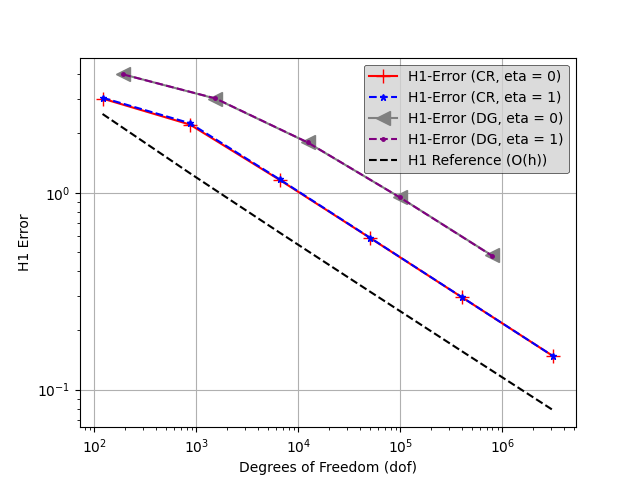}
		\caption{Rate of convergence plot for the numerical solutions $u_h^{CR}$ and  $u_h^{DG}$ for the Caputo fractional time derivative term for the solution defined in Type II   in 2D and 3D .}\label{fig4}
	\end{center}
\end{figure}
	\subsection{Solving GBHE  with Non-Homogenous Boundary Conditions}
	Consider the GBHE with memory defined in \eqref{2.GBHE} on the domain $\Omega = [0,1]^d \times (0,T) $. 
Let $Re$ be the Reynolds number and the kinematic viscosity coefficient is defined as  $\nu := \frac{1}{Re}$. For the 2D case, we set $Re = 50, 100$ and the exact solution \cite{Yta} is taken to be, 
	\begin{align}	 \label{2.sol3}
			u(t,x,y) = \frac{1}{1+e^{\frac{Re(x+y-t)}{2}}}, \quad (x,y)\in \Omega,\quad t\geq 0.
	\end{align}

	where $u$ represents the velocity.
	The initial value, boundary value and the external force $f$ are manufactured by the exact solution \eqref{2.sol3}.
The approximated solution at $T=1$ is shown in Figure \ref{fig8}. It reflects a  notable increase across the line $x + y = 1$ for the Reynolds number $Re = 50$ and $Re= 100.$ The error plots in Figure \ref{fig9} (Left panel) show the convergence rate of $O(h)$ for both the Reynolds numbers for GBHE with and without memory  

Analogously, the computed solution using DGFEM has been demonstrated for the 3D case, where 
\[	u(t,x,y,z) = \frac{1}{1+e^{\frac{Re(x+y+z-t)}{2}}}, \quad (x,y)\in \Omega,\quad t\geq 0. \]
 The solution at $T=1$ is shown in Figure \ref{fig10} and the error plots have been illustrated in Figure \ref{fig9} (Right panel).
	\begin{figure}
	\begin{center}
		\includegraphics[width=0.490\textwidth]{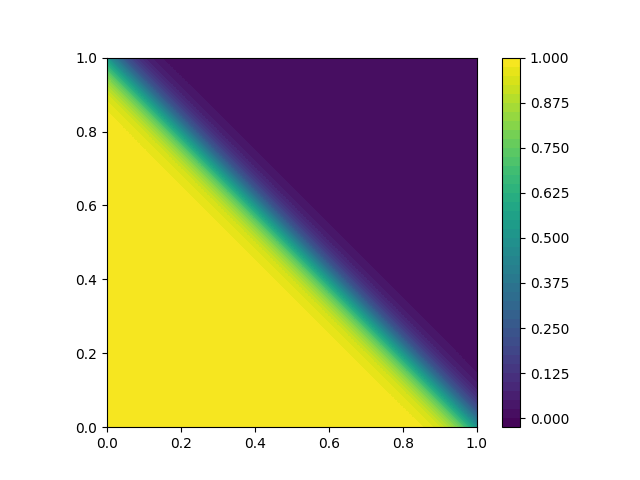}
		\includegraphics[width=0.490\textwidth]{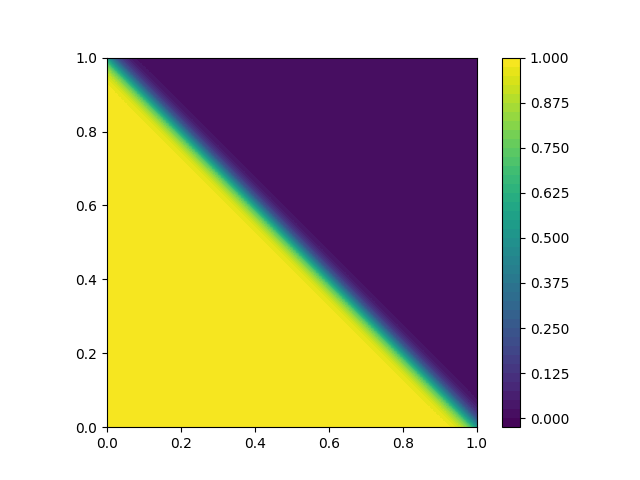}
		\caption{Approximated solution at $T=1,$ for GBHE 2D with memory $(\eta=1)$ with Re = 50 and Re = 100 respectively.}\label{fig8}
	\end{center}
\end{figure}

	\begin{figure}
	\begin{center}
		\includegraphics[width=0.490\textwidth]{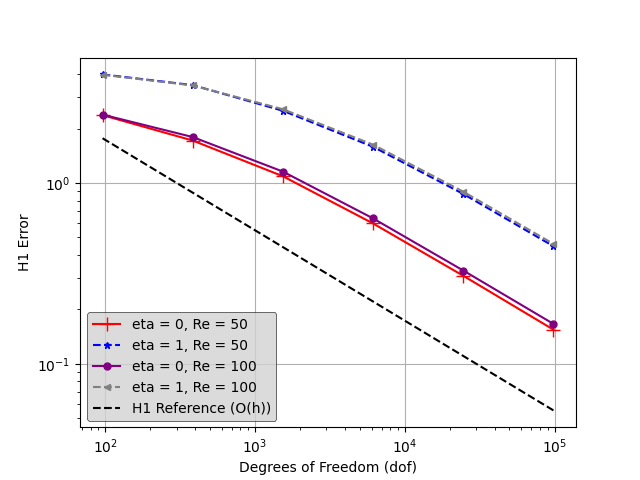}
		\includegraphics[width=0.490\textwidth]{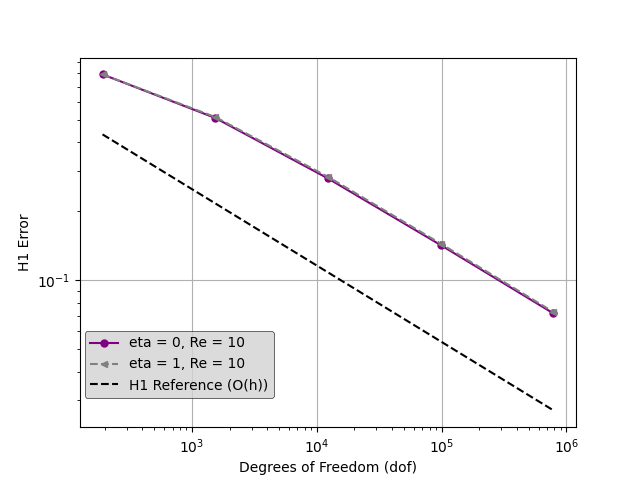}
		\caption{Error plot for GBHE without memory ($\eta = 0$) and with memory $(\eta =1)$ in 2D (left panel) respectively for Re = 50 and 100 and Re = 10 for 3D (right panel) .}\label{fig9}
	\end{center}
\end{figure}
	\begin{figure}
	\begin{center}
		\includegraphics[width=0.400\textwidth]{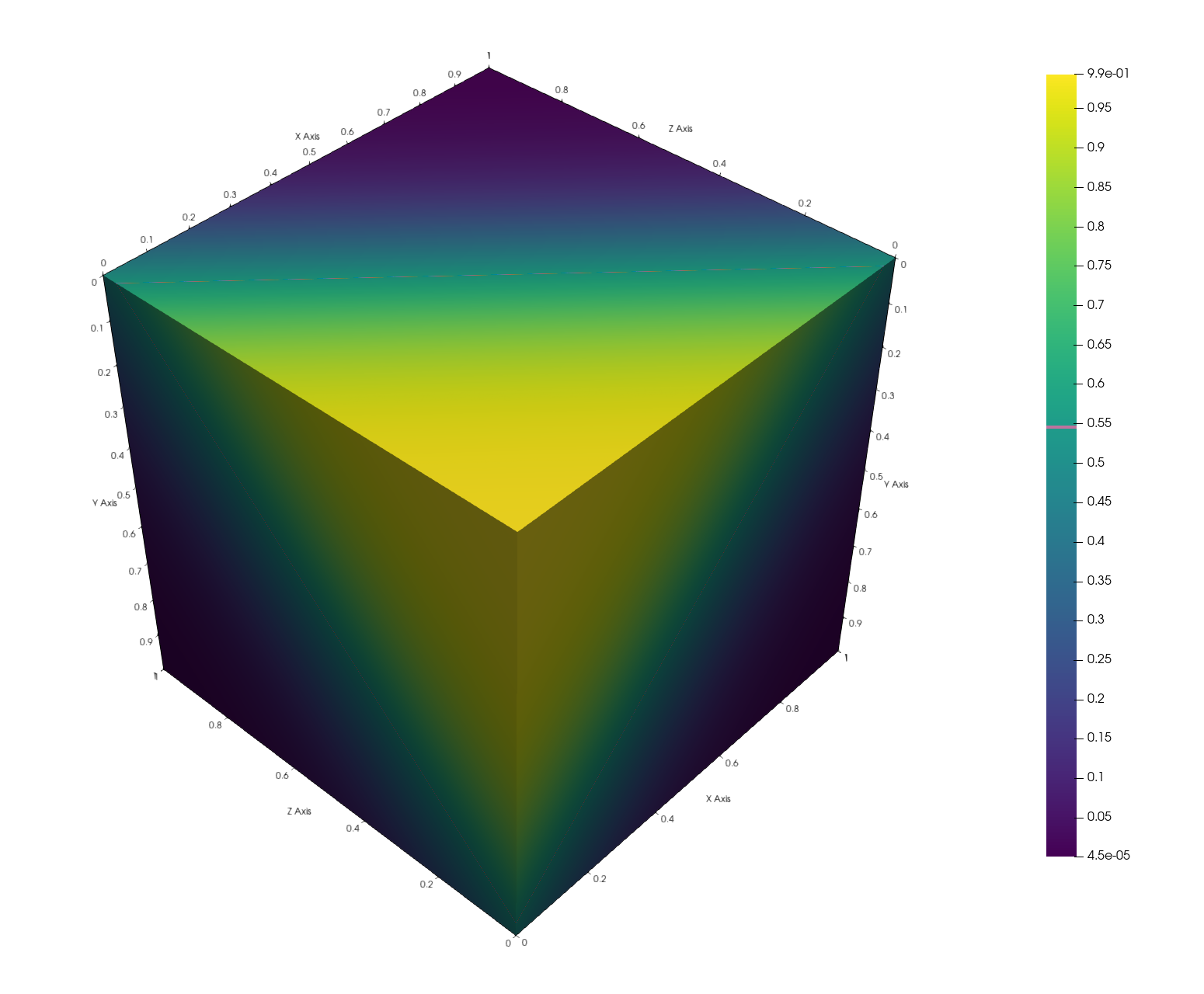}
		\includegraphics[width=0.400\textwidth]{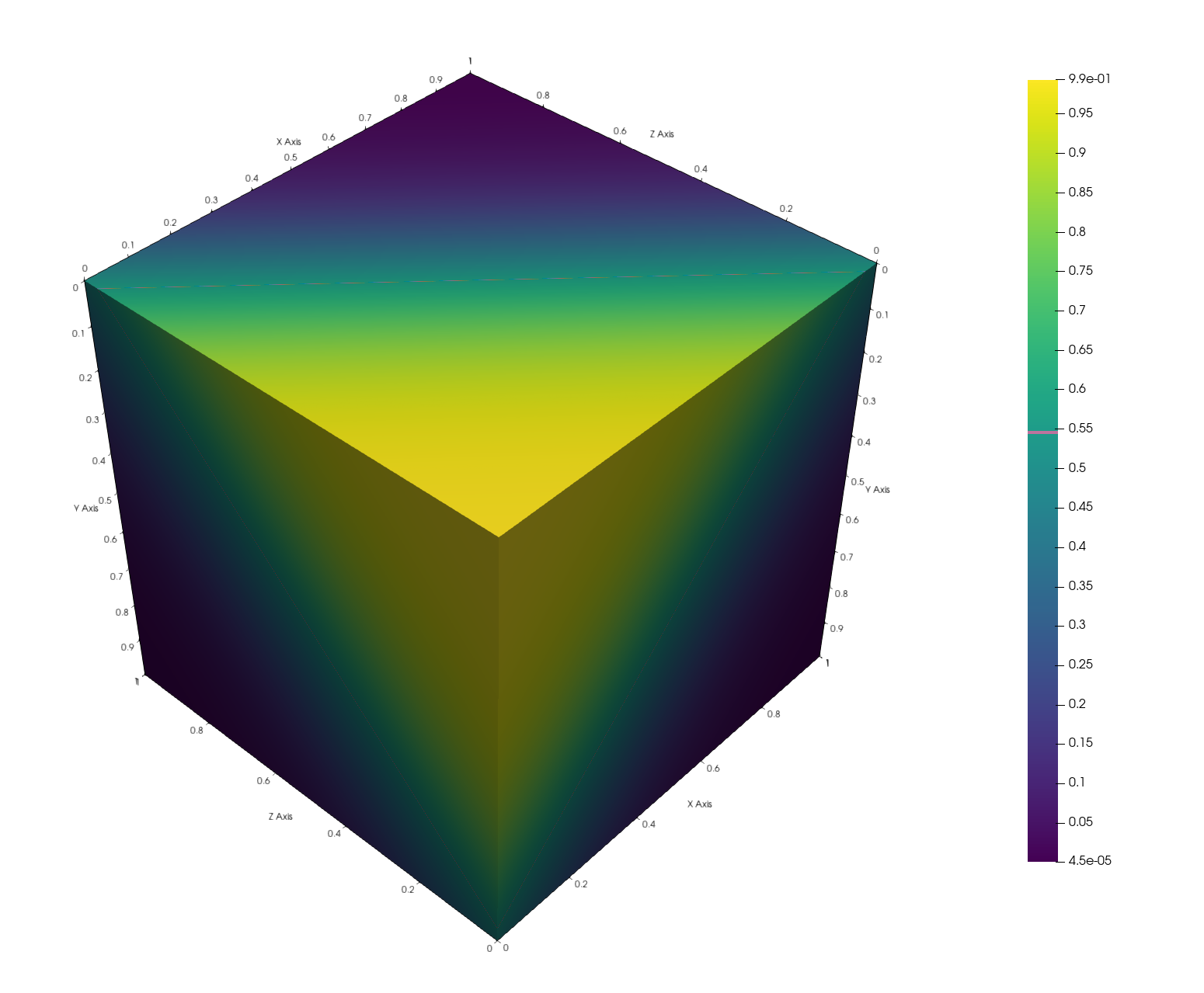}
		\caption{Approximated solution at $T=1,$ for GBHE 3D without memory $(\eta=0)$ and with memory $(\eta =1),$ for $Re=10$ respectively.}\label{fig10}
	\end{center}
\end{figure}
	\subsection{Spiral Wave Formulation}
	In the last example,  a nonlinear system of model having applications in the transmission of electrical impulses in a nerve axon is discussed. The FitzHugh–Nagumo model describes complex wave phenomena in oscillatory media and can be obtained from GBHE ($\alpha =0,$ $\delta=1$ and $\eta=0$) coupled with an ODE as given in \cite{KSP}. In the similar context, GBHE with memory reads:
	\begin{equation}\label{2.TBH}
		\mathcal{L} u(x,t)+ v(x,t) = 0, \qquad \partial_t v(x,t) = \varepsilon(u(x,t)-\rho v(x,t)),
	\end{equation}
where $	\mathcal{L} u(x,t)$ is as defined in \eqref{2.GBHE}, and the parameters $\epsilon$ and $\rho$ represents different scales of the physical variables. 

The weak form similar to \eqref{2.dgfdweakform} for the DGFEM can be obtained and the computed results are presented in Figure \ref{2.fig:ex3} on the domain $\Omega = (0,300)^2$ and other parameters chosen as in \cite{GBHE}. The figures illustrate the spiral behaviour of the solution for the FitzHugh–Nagumo model, GBHE without memory ($\eta = 0$) and GBHE with memory ($\eta = 0.01$). The results illustrate that the addition of the advection term or memory does not affect the spiral behaviour much. However, it is observed that if we increase the memory coefficient $\eta$ to $1$, the spiral behaviour is reversed and the spiral nature is affected if the non-linearity parameter $\delta$ is increased.   

	\begin{figure}[ht!]
		\begin{center}
			\includegraphics[width=0.325\textwidth]{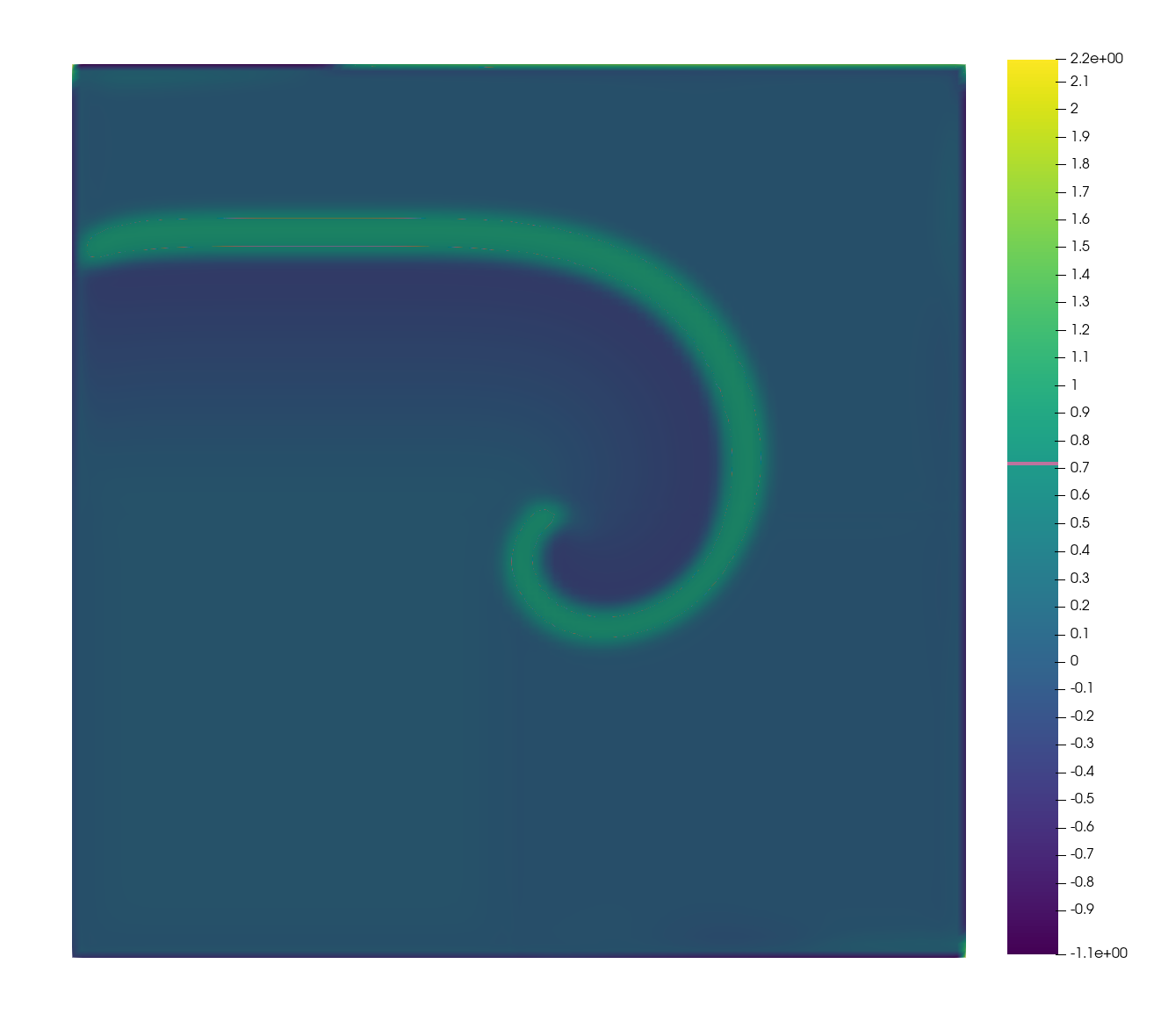}
			\includegraphics[width=0.325\textwidth]{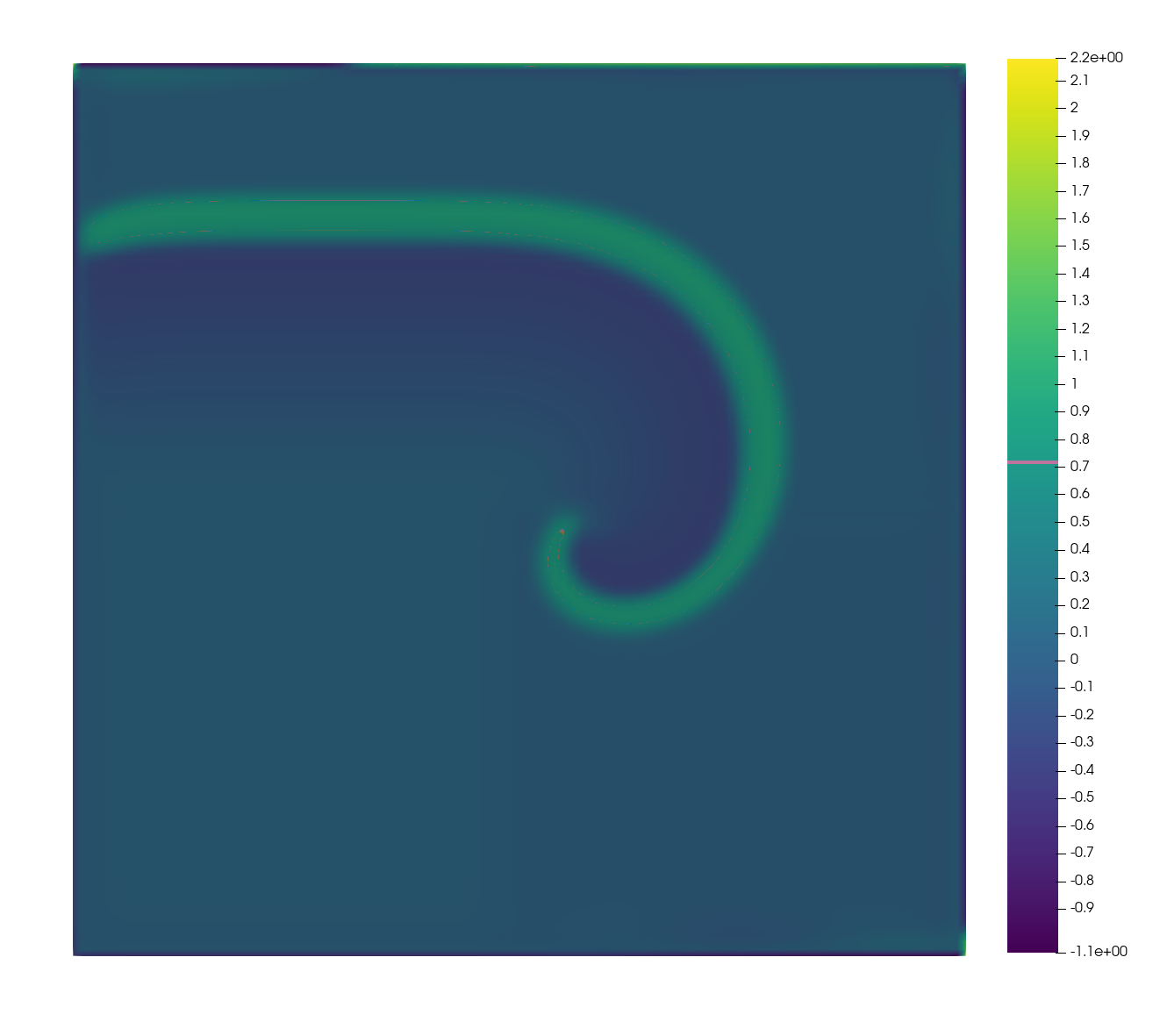}
			\includegraphics[width=0.325\textwidth]{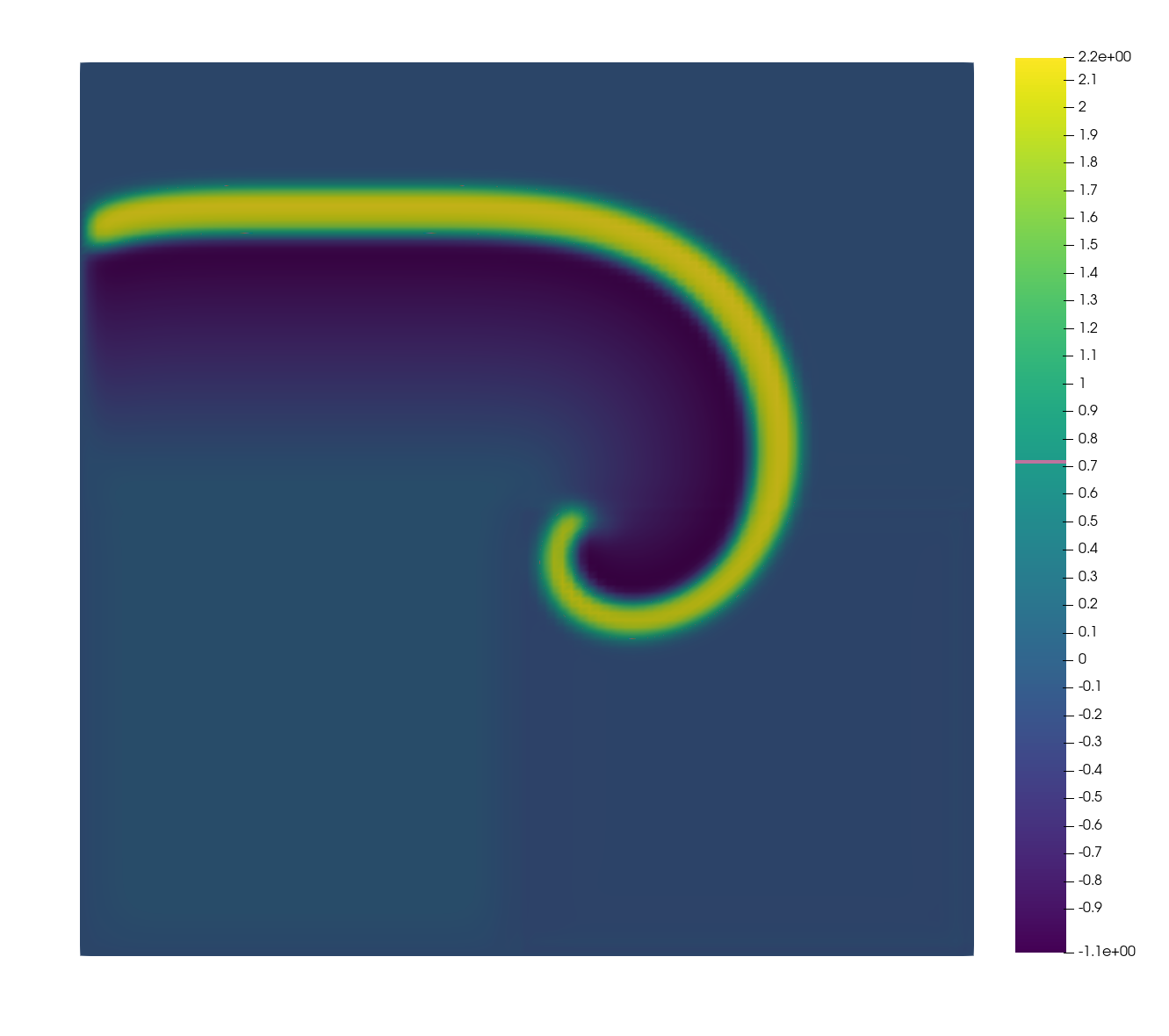}\\
		\end{center}
		\caption{ Snapshots at $t = 150$ of $u_h^{DG}$ for the FitzHugh-Nagumo model, GBHE without memory $(\eta = 0)$ and GBHE with memory $(\eta = 0.01),$ respectively with $\delta = 1$ and $\alpha = 0.1.$}\label{2.fig:ex3}
	\end{figure}
	\section*{Acknowledgement}
	S. Mahajan would like to thank Prof. Manil T. Mohan for useful discussions.
	
\bibliographystyle{siamplain}
\bibliography{IMANUM-refs}

\end{document}